\documentclass[12pt,reqno,letterpaper]{amsart}
\usepackage[T1]{fontenc}
\usepackage[utf8]{inputenc}
\usepackage{lmodern}
\usepackage[letterpaper,left=1in,right=1in,top=0.95in,bottom=1in,headheight=14pt,headsep=18pt]{geometry}
\usepackage{amsmath,amssymb,mathtools,mathrsfs}
\usepackage{microtype,enumitem}
\usepackage{booktabs,array,longtable}
\usepackage[hidelinks,pdfencoding=auto]{hyperref}
\hypersetup{pdftitle={Leafwise Positive Scalar Curvature and the Rosenberg Index},pdfauthor={Guangxiang Su and Zelin Yi}}
\numberwithin{equation}{section}
\allowdisplaybreaks[2]
\theoremstyle{plain}
\newtheorem{theorem}{Theorem}[section]
\newtheorem{proposition}[theorem]{Proposition}
\newtheorem{lemma}[theorem]{Lemma}
\newtheorem{corollary}[theorem]{Corollary}
\theoremstyle{definition}
\newtheorem{definition}[theorem]{Definition}
\theoremstyle{remark}
\newtheorem{remark}[theorem]{Remark}
\newcommand{\M}{\mathcal M}
\newcommand{\LF}{\mathcal F}
\newcommand{\Ecal}{\mathcal E}
\newcommand{\Hcal}{\mathscr H}
\newcommand{\Acal}{\mathscr A}
\newcommand{\Dcal}{\mathscr D}
\newcommand{\Fone}{\mathcal F_1^\perp}
\newcommand{\Ftwo}{\mathcal F_2^\perp}
\newcommand{\C}{\mathbb C}
\newcommand{\R}{\mathbb R}
\newcommand{\Z}{\mathbb Z}

\newcommand{\B}{\mathsf B}
\newcommand{\K}{\mathcal K}
\newcommand{\hct}{\widehat\otimes}
\newcommand{\hc}{\widehat c}
\newcommand{\Fc}{\mathsf F}
\newcommand{\dd}{\mathrm d}
\newcommand{\vol}{\operatorname{vol}}
\newcommand{\Tr}{\operatorname{Tr}}
\newcommand{\Sym}{\operatorname{Sym}}
\newcommand{\End}{\operatorname{End}}

\newcommand{\Dom}{\operatorname{Dom}}
\newcommand{\supp}{\operatorname{supp}}

\newcommand{\Ind}{\operatorname{Ind}}
\newcommand{\rank}{\operatorname{rank}}
\newcommand{\grad}{\operatorname{grad}}
\newcommand{\sgn}{\operatorname{sgn}}

\newcommand{\ip}[2]{\langle #1,#2\rangle}
\newcommand{\norm}[1]{\lVert #1\rVert}

\newcommand{\Dsub}{D_{S(\Ftwo),\beta,\varepsilon}}
\newcommand{\Dtotal}{D^{\M}_{\beta,\varepsilon}}
\newcommand{\Dbase}{D^M_{\beta,\varepsilon}}
\title[Leafwise positive scalar curvature and the Rosenberg index]{Leafwise Positive Scalar Curvature and the Rosenberg Index}
\author{Guangxiang Su}
\address{Chern Institute of Mathematics and LPMC, Nankai University, Tianjin 300071, P. R. China}
\email{guangxiangsu@nankai.edu.cn}
\author{Zelin Yi}
\address{School of Mathematical Sciences, Tongji University, Shanghai 200092, P. R. China}
\email{zelin@tongji.edu.cn}
\subjclass[2020]{Primary 53C12, 53C27, 19K35; Secondary 58J20}
\keywords{Leafwise positive scalar curvature, Rosenberg index, Connes fibration, sub-Dirac operator, analytic localization, Kasparov product}
\begin{document}
\begin{abstract}
Let $(M,F)$ be a closed foliated spin manifold. We prove that the complex Rosenberg index of $M$ vanishes if $F$ admits a metric of leafwise positive scalar curvature. We represent a Bott--Dirac Kasparov product by a deformed sub-Dirac operator on the Connes fibration and compute its pairing with the universal flat bundle in two ways. A compactly supported deformation gives an exact tubular model, whose Gaussian low-energy block is the Dirac operator on $M$. Bismut--Lebeau analytic localization shows that the complementary block contributes zero to the pairing, so the product is the Rosenberg index. For the deformed sub-Dirac operator, the rescaled Lichnerowicz formula and the radial potential give a global spectral gap after flat twisting. The same product therefore vanishes.
\end{abstract}
\maketitle

\section{Introduction}\label{sec:intro}
Let $M$ be a closed, even-dimensional spin Riemannian manifold with spinor bundle $S(TM)$, and let $E_0\to M$ be a finite-rank unitary flat vector bundle. The classical Lichnerowicz formula gives
\[
 (D^{E_0})^2=-\Delta^{S(TM)\otimes E_0}+\frac{k^M}{4},
\]
where $k^M$ is the scalar curvature of $M$. If $k^M$ is everywhere positive, then $D^{E_0}$ is invertible, and its Fredholm index vanishes. The Atiyah--Singer index theorem therefore yields the classical vanishing theorem of Lichnerowicz \cite{Lic63}.
\begin{theorem}\label{thm:lich-classic}
If $M$ admits a metric of positive scalar curvature, then
\[
 \ip{\widehat A(TM)\operatorname{ch}(E_0)}{[M]}=0
\]
for every finite-rank unitary flat vector bundle $E_0\to M$.
\end{theorem}
The Chern character of a finite-rank flat bundle vanishes in positive degrees, so the characteristic number in Theorem~\ref{thm:lich-classic} is its rank times the $\widehat A$-genus. To retain more information from flat coefficients, one passes to the Rosenberg index \cite{Ros83}, a higher index with values in the $K$-theory of the group $C^*$-algebra of the fundamental group.

Throughout the paper, $C^*\Gamma$ denotes the full complex group $C^*$-algebra of a discrete group $\Gamma$. The canonical quotient homomorphism gives the corresponding conclusion for the reduced index. Let $\widetilde M$ be the universal cover of $M$, regarded as a principal $\pi_1(M)$-bundle, and let
\[
 E=\widetilde M\times_{\pi_1(M)}C^*\pi_1(M)
\]
be the universal flat Hilbert $C^*\pi_1(M)$-module bundle. Its fibers are right Hilbert modules, and the group acts by left multiplication by the canonical unitaries. We use both the Kasparov-product description of the index and its realization by the flatly twisted Dirac operator.
\begin{definition}\label{def:rosKK}
Fix the normalizing function $\Fc(x)=x(1+x^2)^{-1/2}$. The bounded transform of the Dirac operator determines the class
\begin{equation}\label{eq:baseclass}
 [L^2(M,S(TM)),\Fc(D)]\in KK(C(M),\C).
\end{equation}
The universal flat bundle determines the class
\begin{equation}\label{eq:bundleclass}
 [C(M,E),0]\in KK(\C,C(M)\otimes C^*\pi_1(M)).
\end{equation}
The Rosenberg index of $M$ is the Kasparov product of \eqref{eq:bundleclass} with the exterior product of \eqref{eq:baseclass} and $1_{C^*\pi_1(M)}$.
\end{definition}
Equivalently, the Rosenberg index is the higher index of the Dirac operator with coefficients in the universal flat bundle.
\begin{definition}[\cite{Ros86}]\label{def:rosIndex}
Let $D^E$ be the Dirac operator twisted by $E$. The Rosenberg index of an even-dimensional manifold $M$ is
\[
 \alpha(M)=\Ind(D^E)\in K_0(C^*\pi_1(M)).
\]
If $M$ is odd-dimensional, then $M\times S^1$ is a closed, even-dimensional spin manifold, and
\[
 K_0\bigl(C^*(\pi_1(M)\times\Z)\bigr)
 \simeq K_0(C^*\pi_1(M))\otimes1\ \oplus\ K_1(C^*\pi_1(M))\otimes e,
\]
where $1$ and $e$ are the standard generators for the circle factor. The odd-dimensional Rosenberg index is characterized by
\begin{equation}\label{eq:odddefinition}
 \alpha(M)\boxtimes e=\alpha(M\times S^1).
\end{equation}
\end{definition}
The displayed splitting follows from the split evaluation sequence for $C(S^1)$ and suspension. In particular, the map $x\mapsto x\boxtimes e$ is injective. No general K\"unneth hypothesis on $C^*\pi_1(M)$ is needed.

The higher-index form of the Lichnerowicz obstruction is the following theorem of Rosenberg.
\begin{theorem}[\cite{Ros86}]\label{thm:rosclassic}
If $M$ admits a metric of positive scalar curvature, then its Rosenberg index vanishes.
\end{theorem}
\begin{remark}\label{rem:even}
By the definition of the Rosenberg index in odd dimensions, it suffices to prove the corresponding vanishing statement for even-dimensional manifolds.
\end{remark}
We now turn to the foliated setting. Let $F\subset TM$ be an integrable subbundle. A metric on $F$ has leafwise positive scalar curvature if its restriction to each leaf has everywhere positive scalar curvature. Connes first obtained an index-theoretic obstruction when $F$ is spin \cite{Con86}; Zhang subsequently extended the result to the case in which the ambient manifold is spin \cite{Zha17}.
\begin{theorem}\label{thm:conneszhang}
Let $F$ be an integrable subbundle of the tangent bundle of a closed, oriented, smooth manifold $M$. If $F$ or $TM$ is spin and $F$ admits a metric of leafwise positive scalar curvature, then $\ip{\widehat A(TM)}{[M]}=0$.
\end{theorem}
\subsection{The main theorem}
We prove the higher-index analogue of Theorem~\ref{thm:conneszhang} when the ambient manifold is spin.
\begin{theorem}\label{thm:main}
If a closed, foliated spin manifold $M$ admits a metric of leafwise positive scalar curvature, then its complex Rosenberg index vanishes.
\end{theorem}
The theorem requires no separate spin or orientability assumption on $F$ or $TM/F$. The trivial representation of $\pi_1(M)$ induces a homomorphism $C^*\pi_1(M)\to\C$ and hence a map on $K$-theory. In even dimensions this map sends the Rosenberg index to the $\widehat A$-genus
\[
 \ip{\widehat A(TM)}{[M]}.
\]
Thus Theorem~\ref{thm:main} recovers Zhang's $\widehat A$-genus vanishing result for spin ambient manifolds \cite[Theorem~0.1]{Zha17}. By Remark~\ref{rem:even}, it is enough to prove the main theorem in even dimensions. For the related case in which $F$ is spin and $M$ need not be spin, \cite{SY25} defines a foliated Rosenberg index and proves that it obstructs leafwise positive scalar curvature.

\subsection{Strategy of the proof}
We compute the pairing of the universal flat bundle with a deformed class on the Connes fibration. The deformed sub-Dirac operator represents the Kasparov product of the fiberwise Bott class and the Dirac class of the total space. The proof identifies this pairing with the Rosenberg index and then shows that it vanishes.

For the identification, we first deform the operator in a compact neighborhood of the distinguished section to obtain an exact tubular model. For the metric defining the sub-Dirac operator, the section need not be orthogonal to the vertical tangent bundle. The construction therefore includes the change of volume density and the zeroth-order connection correction. A common-domain argument shows that this deformation preserves the $KK$-class. Bismut--Lebeau localization then identifies the Gaussian low-energy block with the Dirac operator on $M$. The off-diagonal terms are bounded and exponentially small, while the complementary block has a spectral gap of order $T^{1/2}$. This growing gap makes its product with the universal flat bundle zero, leaving precisely the Rosenberg index.

For vanishing, we use the deformed sub-Dirac operator. After twisting by the universal flat bundle, the rescaled Lichnerowicz formula gives a lower bound near the distinguished section, and the radial potential gives a lower bound at infinity. A suitable choice of parameters combines them into a global spectral gap. The twisted operator is therefore invertible, and the same product vanishes. Only this second calculation uses leafwise positive scalar curvature; the auxiliary deformation serves to identify the class.

\subsection{Organization of the paper}
Section~\ref{sec:setup} introduces the Connes fibration, the rescaled metrics, and the Kasparov modules. Sections~\ref{sec:elliptic} and \ref{sec:geometry} establish the analytic and geometric estimates needed on the noncompact total space. Section~\ref{sec:cycle} constructs the deformed cycle and identifies its Kasparov product. Sections~\ref{sec:localization} and \ref{sec:identification} carry out normalization and localization to identify the flat-coefficient product with the Rosenberg index. Section~\ref{sec:vanishing} proves that this product vanishes under leafwise positive scalar curvature. Section~\ref{sec:mainproof} assembles these results to prove the main theorem.

\
\\
\
{\bf Use of AI}\ \ AI was used to check the content and polish the language after the manuscript was completed. The authors take full responsibility for the correctness of the paper..

\section{Geometric and KK-theoretic setup}\label{sec:setup}
The Connes fibration carries both the lifted foliation and a radial potential that confines sections near the distinguished copy of $M$. We choose its geometry so that the rescaled curvature formula applies near this section and commutators with base functions remain bounded at infinity. These choices define the Bott and Dirac classes whose product will be computed in the rest of the paper.
\subsection{The Connes fibration}
Let $(M,F)$ be a closed foliated spin manifold of even dimension $m$, and let $k=\rank(TM/F)$. The case $k=0$ follows from Theorem~\ref{thm:rosclassic}, so we assume $k>0$. The Connes fibration \cite{Con86} is the bundle $p:\M\to M$ whose fiber over $x\in M$ consists of the Euclidean metrics on $T_xM/F_x$. Its typical fiber is
\[
 GL(k,\R)/O(k)\simeq GL^+(k,\R)/SO(k).
\]
The first description does not require an orientation on $TM/F$. At the identity coset, the tangent space of the symmetric space is the vector space $\mathfrak p$ of real symmetric matrices. We equip each fiber with the invariant Riemannian metric determined on $\mathfrak p$ by
\begin{equation}\label{eq:trace}
 \ip XY=\Tr(XY).
\end{equation}
Choosing a Euclidean metric $g^{TM/F}$ on $TM/F$ gives a distinguished section and hence an embedding $s:M\hookrightarrow\M$.

The Bott connection lifts $F$ to an integrable distribution $\LF$ on $\M$, and the metric $g^F$ lifts to a metric $g^{\LF}$; see \cite{Gro96,Zha17}. Indeed, the Bott partial connection is
\[
 \nabla^{\mathrm B}_X\overline U=\overline{[X,U]},\qquad X\in\Gamma(F),
\]
and its leafwise curvature vanishes by the Jacobi identity. Let $\Ftwo\to\M$ denote the vertical tangent bundle of $p$. The fiber metric induced by \eqref{eq:trace} defines a metric $g^{\Ftwo}$ on this bundle. It remains to choose a horizontal complement to $\LF\oplus\Ftwo$.

\subsection{The horizontal complement \texorpdfstring{$\Fone$}{F1}}
To lift transverse directions, fix a decomposition $TM=F\oplus F^\perp$ and identify $F^\perp$ with $TM/F$. Choose a connection $\nabla^{TM/F}$ compatible with $g^{TM/F}$. Its parallel transport defines a horizontal lift to the bundle of Euclidean metrics. More explicitly, let $f$ be a local section of $F^\perp$, let $\gamma:[0,1]\to M$ be an integral curve of $f$, and fix $z\in p^{-1}(\gamma(0))$. Choose a $z$-orthonormal basis
\begin{equation}\label{eq:basis}
 s_1(\gamma(0)),s_2(\gamma(0)),\ldots,s_k(\gamma(0))
\end{equation}
of $T_{\gamma(0)}M/F_{\gamma(0)}$, and parallel-transport it along $\gamma$. A basis determines the unique Euclidean metric for which it is orthonormal. Thus the curve of bases
\begin{equation}\label{eq:liftcurve}
 t\longmapsto (s_1(\gamma(t)),\ldots,s_k(\gamma(t)))
\end{equation}
defines a curve in $\M$ starting at $z$. Its tangent at $t=0$ is the lift $f^H_z\in T_z\M$, characterized by this parallel transport and $p_*f^H_z=f_{\gamma(0)}$. Varying $z$ and $f$ gives the rank-$k$ bundle $\Fone$. Thus $\Fone$ is defined by $\nabla^{TM/F}$, whereas the lift of $F$ is defined by the Bott connection.

For $z\in\M$, let $\ell(z)$ be the length of the unique fiberwise geodesic from $s(p(z))$ to $z$. The chosen horizontal parallel transport preserves both the fiberwise symmetric-space metric and the distinguished section.
\begin{proposition}\label{prop:transverse}
We have $f^H\ell=0$ and $\mathcal L_{f^H}g^{\Ftwo}=0$ for every local section $f$ of $F^\perp$.
\end{proposition}
\begin{proof}
Let $P_t:T_{\gamma(0)}M/F_{\gamma(0)}\to T_{\gamma(t)}M/F_{\gamma(t)}$ be parallel transport for $\nabla^{TM/F}$. Because this connection is metric, $P_t$ sends $g^{TM/F}_{\gamma(0)}$ to $g^{TM/F}_{\gamma(t)}$. Its action on the space of Euclidean metrics is an isometry for \eqref{eq:trace}, by invariance of the symmetric-space metric under $GL(k,\R)$. It therefore preserves distance from the distinguished metric. Differentiating these two invariances gives the two assertions. This argument applies to every initial metric, not only the distinguished one.
\end{proof}
As in \cite{Con86,Zha17}, there is a canonical metric on $\Fone$ for which $(\M,\LF)$ is almost isometric. A point $z\in\M$ is itself a Euclidean metric on $T_{p(z)}M/F_{p(z)}$, and the canonical metric is
\begin{equation}\label{eq:canonical}
 g^{\Fone}_{\mathrm{cano},z}(f^H_z,g^H_z)=z(f,g).
\end{equation}
We need the canonical metric near the distinguished section for the rescaled curvature estimate. At infinity, we need a transverse metric that is uniformly comparable to a fixed base metric. Inspired by \cite[Equation~(3.16)]{YZ18}, we interpolate between these two choices. Fix $R_0>1$; Section~\ref{sec:vanishing} will specify its value. There is a unique $g^{TM/F}$-self-adjoint endomorphism $S_z$ such that
\[
 g^{\Fone}_{\mathrm{cano},z}(f^H_z,g^H_z)=g^{TM/F}(e^{S_z}f,g).
\]
\begin{definition}\label{def:capmetric}
Let $\psi:\R\to[0,1]$ be smooth, with $\psi=1$ on $[0,2R_0]$ and $\psi=0$ on $[3R_0,\infty)$. Define the modified transverse metric by
\begin{equation}\label{eq:capmetric}
 g^{\Fone}_z(f^H,g^H)=g^{TM/F}(e^{\psi(\ell(z))S_z}f,g).
\end{equation}
\end{definition}
By construction, $g^{\Fone}$ coincides with the canonical metric on $\{\ell\leq2R_0\}$ and with the pullback of $g^{TM/F}$ outside $\{\ell<3R_0\}$. Near the distinguished section this preserves the almost-isometric geometry needed for Zhang's rescaled Lichnerowicz formula \cite[Equation~(1.71)]{Zha17}. At infinity, the pullback form gives the bounded commutators used in Sections~\ref{sec:geometry} and \ref{sec:cycle}. For fixed $R_0$ the modified metric is uniformly equivalent, on all of $\M$, to the pullback metric. The equivalence constants may depend on $R_0$.

\subsection{Rescaled metrics, spin stabilization, and fiberwise coordinates}
The preceding choices give the decomposition
\begin{equation}\label{eq:splitting}
 T\M=\LF\oplus\Fone\oplus\Ftwo.
\end{equation}
For $\beta,\varepsilon>0$, declare these summands orthogonal and equip $T\M$ with
\begin{equation}\label{eq:rescaledmetric}
 g^{T\M}_{\beta,\varepsilon}
 =\beta^2g^{\LF}\oplus\varepsilon^{-2}g^{\Fone}\oplus g^{\Ftwo}.
\end{equation}
Its Levi--Civita connection is denoted by $\nabla^{T\M,\beta,\varepsilon}$. When $\beta=\varepsilon=1$, we abbreviate the metric and connection to $g^{T\M}$ and $\nabla^{T\M}$.

The Bott class below uses the spinor bundle of $\Ftwo$, so we first arrange for the vertical bundle to be spin and even-dimensional. Replace the Connes fibration by the fiber product of four copies. Equivalently, its points are metrics on $(TM/F)^{\oplus4}$ making the four summands orthogonal. The lifted distribution $\LF$ is the diagonal Bott lift, and the transverse canonical metric is defined using the first metric factor. The vertical bundle is the direct sum of the four vertical bundles. We continue to use $\M,p,s,\Ftwo$ for the stabilized objects.

To verify the spin assertion globally, put
\[
 W=\Sym(TM/F,g^{TM/F}),\qquad N=W^{\oplus4}.
\]
The exponential maps of the fibers identify the stabilized total space with $N$. Moreover,
\[
 w_1(N)=4w_1(W)=0,
 \qquad w_2(N)=4w_2(W)+\binom42w_1(W)^2=0.
\]
Thus $N$ is orientable and spin. Its rank
\[
 d=4\frac{k(k+1)}2=2k(k+1)
\]
is even. Choose a spin structure on $N$ and its pullback on $\Ftwo\simeq p^*N$. Since $\LF\oplus\Fone\simeq p^*TM$, the chosen spin structure on $TM$ gives compatible spin structures on the horizontal bundle and $T\M$. Neither $F$ nor $TM/F$ needs to be spin.

Let $n=m+d=\dim\M$. The projected vertical connection is
\[
 \nabla^{\Ftwo}=p^{\Ftwo}\nabla^{T\M,\beta,\varepsilon}p^{\Ftwo}.
\]
The Koszul calculation in Section~\ref{sec:geometry} shows that this connection is independent of $\beta,\varepsilon$ and the transverse cutoff radius. We next introduce fiberwise coordinates for the radial potential and localization. Identify $N$ with $\Ftwo|_{s(M)}$. In a local trivialization over $O\subset s(M)$, write
\begin{equation}\label{eq:coords}
 (y,Z)=(y_1,\ldots,y_m,Z_{m+1},\ldots,Z_n).
\end{equation}
The fiberwise exponential map gives
\begin{equation}\label{eq:fiberexp}
 \Phi:N\longrightarrow\M,\qquad (y,Z)\longmapsto\exp_y^{\mathrm{fib}}(Z).
\end{equation}
Each fiber is a complete simply connected nonpositively curved symmetric space, so this is a global fiber-preserving diffeomorphism. In matrix notation it is
\[
 (x,Z_1,\ldots,Z_4)\longmapsto
 \bigl(g^{TM/F}_x(e^{2Z_1}\cdot,\cdot),\ldots,
       g^{TM/F}_x(e^{2Z_4}\cdot,\cdot)\bigr).
\]
In particular $\ell(y,Z)=|Z|$. These coordinates use the exponential map within each fiber. The embedded section $s(M)$ need not be orthogonal to the vertical bundle for the ambient metric; this distinction will enter the normalization in Section~\ref{sec:localization}.

For $z\notin s(M)$, let $v(z)$ be the outward unit tangent to the fiberwise geodesic from $s(p(z))$ to $z$. The section $V=\ell v$, extended by zero on $s(M)$, is smooth. In the coordinates \eqref{eq:fiberexp},
\begin{equation}\label{eq:radial}
 V(y,Z)=\sum_{a=m+1}^{n}Z_a\partial_{Z_a}=Z.
\end{equation}
Although $\ell$ is not smooth on $s(M)$, $\ell^2$ and $\sqrt{1+\ell^2}$ are smooth. Statements involving $\dd\ell$ are understood off $s(M)$, or as the corresponding Lipschitz bounds.

\subsection{The relevant Kasparov modules}
The first Kasparov module is the fiberwise Bott class, defined by the radial direction. Pullback along $p$ gives a homomorphism $C(M)\to C_b(\M)$ and hence a $C(M)$-action on $C_0(\M)$. We write $f$ for multiplication by $p^*f$. Let $\chi_v:\M\to[0,1]$ be a smooth radial function which is one for $\ell\geq1$ and zero near $s(M)$. Then $\widetilde v=\chi_vv$, extended by zero, is a smooth vertical section.

Let $S^*(\Ftwo)$ be the graded complex-linear dual of $S(\Ftwo)$, with its induced Hermitian metric. For a vertical vector $w$, define
\begin{equation}\label{eq:dualaction}
 (\hc^*_{\mathrm{spin}}(w)\varphi)(s)
 =-(-1)^{|\varphi|}\varphi(c(w)s)
\end{equation}
for homogeneous $\varphi$. We use $c(w)^*=-c(w)$ and $c(w)^2=-|w|^2$. The dual action is odd and self-adjoint, and its square is $|w|^2$. The pair
\begin{equation}\label{eq:bottclass}
 b=\bigl[C_0(\M,S^*(\Ftwo)),-\hc^*_{\mathrm{spin}}(\widetilde v)\bigr]
 \in KK(C(M),C_0(\M))
\end{equation}
is independent of the cutoff. Indeed, its square differs from one by a compactly supported bundle endomorphism, and it commutes with base functions; see also \cite[Section~5]{Con86}.

The second module is the spin Dirac class of the total space. Let $\Dtotal$ be the spin Dirac operator for \eqref{eq:rescaledmetric}. Completeness and essential self-adjointness are proved in Corollary~\ref{cor:complete}. Thus
\begin{equation}\label{eq:totalclass}
 d_{\M}=\bigl[L^2(\M,S(T\M)),\Fc(\Dtotal)\bigr]
 \in KK(C_0(\M),\C)
\end{equation}
is the spin Dirac class; see \cite[Theorem~10.6.6]{HR00}. It is independent of the chosen complete metric representing the fixed spin structure, in particular of the positive rescaling parameters.

To represent the product of these classes, we use the Liu--Zhang sub-Dirac operator \cite{LZ01,Zha17}, denoted by $\Dsub$. It acts on
\[
 \Ecal=S(\LF\oplus\Fone)\hct\Lambda^*\Ftwo\longrightarrow\M.
\]
It is the spin Dirac operator on $S(T\M)$ twisted by $S^*(\Ftwo)$ and transported to the displayed exterior-algebra bundle. We denote its connection by $\nabla^{\Ecal}$; under this identification it is induced by the tensor-product connection $\nabla^{S(T\M)}\hct1+1\hct\nabla^{S^*(\Ftwo)}$.

\subsection{The deformed product}
Adding the radial potential to $\Dsub$ confines sections near $s(M)$ and incorporates the Bott class into the operator. For every $T>0$, the product of \eqref{eq:bottclass} and \eqref{eq:totalclass} is represented by
\begin{equation}\label{eq:deformedclass}
 \bigl[L^2(\M,\Ecal),\Fc(\Dsub+T\hc(V))\bigr].
\end{equation}
This is established in Proposition~\ref{prop:product}. On the exterior-algebra factor,
\[
 c(w)=\epsilon(w)-\iota(w),\qquad
 \hc(w)=\epsilon(w)+\iota(w),
\]
with the graded tensor convention on $\Ecal$. In particular $\hc(V)$ is odd and self-adjoint, $\hc(V)^2=|V|^2$, and it anticommutes with every tangent Clifford multiplication. The class of \eqref{eq:deformedclass} is independent of $T>0$ and of the rescaling parameters.

\subsection{The two analytic tasks}
First, analytic localization \cite{BL91} identifies an auxiliary representative of \eqref{eq:deformedclass}, obtained by a compactly supported homotopy, with
\begin{equation}\label{eq:decomposition}
 [L^2(M,S(TM)),\Fc(\Dbase)]\ \oplus\ [H'_T,\Fc(A_{T,4})].
\end{equation}
The second block is self-adjoint with compact resolvent. Its Kasparov product with the universal flat bundle vanishes for sufficiently large $T$. Consequently, after the coefficient exterior products are understood, the threefold product
\begin{equation}\label{eq:triple}
 [C(M,E),0]\hct_{C(M)\otimes C^*\pi_1(M)}
 \bigl((b\hct_{C_0(\M)}d_{\M})\otimes1_{C^*\pi_1(M)}\bigr)
\end{equation}
is the Rosenberg index of $M$.

Section~\ref{sec:localization} defines the auxiliary operator and its four blocks. For the second task, we use the deformed sub-Dirac operator in \eqref{eq:deformedclass}: when the leafwise scalar curvature is positive, a suitable choice of $\beta,\varepsilon,T$ makes its flat twist coercive and invertible. The first task identifies \eqref{eq:triple} with $\alpha(M)$, and the second proves that \eqref{eq:triple} is zero.
\section{Elliptic estimates on noncompact manifolds}\label{sec:elliptic}
The deformed operator acts on the noncompact total space $\M$. To prove compactness of its resolvent, we will combine elliptic regularity on compact sets with an estimate controlling sections at infinity. We record the local analytic results here, followed by a bounded-perturbation lemma used to compare representatives of the index class.

Let $X$ be a noncompact manifold, and let $W\to X$ be a Hermitian vector bundle equipped with a connection $\nabla$. For an integer $j\geq0$, define $H^j(X,W)$ as the completion of $C_c^\infty(X,W)$ for the norm
\begin{equation}\label{eq:sobolev}
 \norm u_{H^j}^2=\norm u_{L^2}^2+\norm{\nabla u}_{L^2}^2+\cdots+\norm{\nabla^ju}_{L^2}^2.
\end{equation}
These global Sobolev spaces depend on the metrics and connections. On a fixed compact set, however, smooth choices give equivalent norms.

For compact $K\subset X$, write $H^j(K,W)$ for the completion, in \eqref{eq:sobolev}, of smooth sections supported in $K$. A first-order differential operator is elliptic near $K$ if its principal symbol is invertible at every nonzero cotangent vector in a neighborhood of $K$. We use the following local results.
\begin{proposition}[Rellich lemma; {\cite[Section~10.4.3]{HR00}, \cite[Theorem~2.34]{Aub98}}]\label{prop:rellich}
The inclusion $H^1(K,W)\hookrightarrow L^2(X,W)$ is compact for every compact subset $K\subset X$.
\end{proposition}
\begin{proposition}[G\aa rding inequality; {\cite[Section~10.4.4]{HR00}}]\label{prop:garding}
Let $D$ be a first-order differential operator elliptic on a neighborhood of a compact set $K$. There is a constant $C_K$ such that
\begin{equation}\label{eq:garding}
 \norm u_{H^1(K,W)}\leq C_K\bigl(\norm{Du}_{L^2}+\norm u_{L^2}\bigr)
\end{equation}
for smooth sections supported in $K$, and hence on the corresponding graph domain. More generally, if $\chi,\widetilde\chi$ are compactly supported smooth functions with $\widetilde\chi=1$ near $\supp\chi$, then
\[
 \norm{\chi u}_{H^1}\leq C_{\chi,\widetilde\chi}
 \bigl(\norm{\widetilde\chi Du}_{L^2}+\norm{\widetilde\chi u}_{L^2}\bigr)
\]
for locally defined distributional sections for which the right-hand side is finite.
\end{proposition}
Propositions~\ref{prop:rellich} and \ref{prop:garding} give local $L^2$ precompactness for sequences bounded in the graph norm of an elliptic operator. In Proposition~\ref{prop:deformedcycle}, the confining potential will control the $L^2$ mass outside compact sets and thereby give global compactness.

The next lemma allows us to change an operator by a bounded self-adjoint term without changing its bounded transform modulo compact operators. It applies both to the localization argument and to the later coefficient twist. A regular self-adjoint Hilbert $C^*$-module operator is understood in the usual unbounded-operator sense; its bounded transform and resolvents are defined by continuous functional calculus.
\begin{lemma}\label{lem:boundedperturbation}
Let $R,S$ be regular self-adjoint operators on the same Hilbert space or Hilbert $C^*$-module. Suppose they have the same domain and compact resolvents, and that $K=R-S$ is bounded. Then $\Fc(R)-\Fc(S)$ is compact.
\end{lemma}
\begin{proof}
We express the difference of bounded transforms as an integral whose integrand is compact and has an integrable norm bound. Put $a^2=1+\lambda^2$, $R_a=(a^2+R^2)^{-1}$, and $S_a=(a^2+S^2)^{-1}$. The resolvent identity gives
\begin{equation}\label{eq:boundedperturbation}
 \Fc(R)-\Fc(S)=\frac2\pi\int_0^\infty
 R_a\bigl(a^2K-RKS\bigr)S_a\,\dd\lambda.
\end{equation}
The term containing $RKS$ is interpreted as $(RR_a)K(SS_a)$, a product of bounded operators. Thus the formula requires no domain invariance for $K$. It follows by combining the two resolvent identities at $\pm ia$ and then integrating. Functional calculus gives
\[
 \norm{R_a},\norm{S_a}\leq a^{-2},\qquad
 \norm{RR_a},\norm{SS_a}\leq(2a)^{-1}.
\]
Compactness of the resolvents makes every integrand compact, and the displayed bounds give an integrand norm of at most $5\norm K/(4a^2)$. Hence the integral converges in norm to a compact operator. The same estimate gives norm continuity of the bounded transform along norm-continuous bounded self-adjoint perturbations with a common domain.
\end{proof}

\section{Geometry at infinity and self-adjointness}\label{sec:geometry}
The radial potential must dominate the error term in the square of the deformed operator. We therefore estimate derivatives of the radial field $V$ in the leafwise, transverse, and vertical directions. The transverse derivative vanishes, while the other two estimates are independent of the cutoff radius $R_0$. This independence will allow us to choose $R_0$ later in the positive-curvature argument. The same estimates give completeness and essential self-adjointness. We also verify separately that the modified transverse metric bounds commutators with functions pulled back from $M$.

\subsection{Connections and spinors}
We first specify the connection used in the sub-Dirac operator and identify its action on the vertical bundle. Let $\nabla^{T\M}$ be the Levi--Civita connection of $g^{T\M}$. Using the orthogonal projections associated with \eqref{eq:splitting}, define
\begin{equation}\label{eq:projectedconnections}
 \nabla^{\LF\oplus\Fone}
 =p^{\LF\oplus\Fone}\nabla^{T\M}p^{\LF\oplus\Fone},
 \qquad
 \nabla^{\Ftwo}=p^{\Ftwo}\nabla^{T\M}p^{\Ftwo}.
\end{equation}
These connections induce the corresponding spin connections and the connection on $\Lambda^*\Ftwo$. Via the canonical graded identification
\[
 S(T\M)\hct S^*(\Ftwo)
 \simeq S(\LF\oplus\Fone)\hct\Lambda^*\Ftwo,
\]
the tensor-product connection $\nabla^{S(T\M)}\hct1+1\hct\nabla^{S^*(\Ftwo)}$ induces the connection $\nabla^{\Ecal}$ on the right-hand bundle. For the following comparison, all projected and induced connections are taken with respect to the rescaled metric \eqref{eq:rescaledmetric}; we suppress their parameter labels $\beta,\varepsilon$.

To express $\nabla^{\Ecal}$ in terms of the projected connections for the rescaled metric, put
\[
 \nabla^0=\nabla^{\LF\oplus\Fone}\oplus\nabla^{\Ftwo},\qquad
 \mathcal S=\nabla^{T\M,\beta,\varepsilon}-\nabla^0.
\]
In an adapted orthonormal frame $f_1,\ldots,f_n$ for the rescaled metric, the induced connection $\nabla^{\Ecal}$ is
\[
 \nabla^{\Ecal}_X
 =\nabla^{S(\LF\oplus\Fone)\hct\Lambda^*\Ftwo,0}_X
 +\frac14\sum_{i,j=1}^n
 \ip{\mathcal S(X)f_i}{f_j}_{\beta,\varepsilon}c_{\beta,\varepsilon}(f_i)c_{\beta,\varepsilon}(f_j).
\]
The last term records the mixing of the horizontal and vertical bundles and is included in the sub-Dirac operator. An isometry between the horizontal spinor bundle and a pullback of a base spinor bundle does not, by itself, identify their connections.

The projected vertical connection can be described directly through fiberwise geometry and horizontal transport. Let $X$ be the horizontal lift of a local vector field on $M$ for the distribution $\LF\oplus\Fone$, and let $U,W$ be vertical vector fields. Along $F$, the lift is defined by the Bott connection; along $F^\perp$, it is defined by $\nabla^{TM/F}$. Both transports act by fiber isometries, so
\[
 X\ip UW=\ip{[X,U]}W+\ip U{[X,W]}.
\]
For every $\beta,\varepsilon>0$, apply the Koszul formula to the Levi--Civita connection $\nabla^{T\M,\beta,\varepsilon}$ of $g^{T\M}_{\beta,\varepsilon}$. Since the splitting remains orthogonal and the vertical metric is unchanged, it gives
\[
 \begin{aligned}
 2\ip{\nabla^{T\M,\beta,\varepsilon}_UW}X_{\beta,\varepsilon}
 &=-X\ip UW+\ip{[X,U]}W+\ip U{[X,W]}=0,\\
 p^{\Ftwo}\nabla^{T\M,\beta,\varepsilon}_XU&=[X,U].
 \end{aligned}
\]
Thus the fibers are totally geodesic for every rescaled metric. With all three arguments vertical, the Koszul formula involves only $g^{\Ftwo}$, so the projected connection in vertical directions is the fiber Levi--Civita connection. In horizontal directions, it is given by the bracket formula above. Both descriptions use only the fixed vertical metric and horizontal distribution, which are unchanged by rescaling or by the transverse metric modification in \eqref{eq:capmetric}. Hence $p^{\Ftwo}\nabla^{T\M,\beta,\varepsilon}p^{\Ftwo}$ is the same connection for all $\beta,\varepsilon$ and all cutoff radii $R_0$, as indicated by the notation $\nabla^{\Ftwo}$.

\subsection{The transverse metric at infinity}
For a Dirac-type operator, the commutator with a function is Clifford multiplication by its differential. We therefore need a global bound on the differential of each $p^*f$, where $f\in C^\infty(M)$. Recall that $g^{\Fone}$ agrees with the canonical metric on $\{\ell\leq2R_0\}$ and with a pullback metric outside $\{\ell<3R_0\}$. The pullback form at infinity gives the required bound in transverse directions.
\begin{lemma}\label{lem:basecommutators}
For fixed $\beta,\varepsilon\in(0,1]$ and $f\in C^\infty(M)$, the derivative $X(p^*f)$ is uniformly bounded over all $g^{T\M}_{\beta,\varepsilon}$-unit tangent vectors $X$ on $\M$.
\end{lemma}
\begin{proof}
We bound the derivative on each summand of \eqref{eq:splitting}. Vertical vectors annihilate $p^*f$. If $X$ is a rescaled leafwise unit vector, then $p_*X=\beta^{-1}Y$ for a $g^F$-unit vector $Y$, so $X(p^*f)=\beta^{-1}p^*(Yf)$ is bounded for fixed $\beta$ by compactness of $M$.

For a transverse vector, write $X=Y^H$, with $Y\in F^\perp_{p(z)}$ and the lift defined by $\nabla^{TM/F}$. If $X$ is a rescaled unit vector, then
\[
 \varepsilon^2
 =g^{\Fone}_z(Y^H,Y^H)
 =g^{TM/F}(e^{\psi(\ell(z))S_z}Y,Y).
\]
Outside $\{\ell<3R_0\}$ the last expression is $g^{TM/F}(Y,Y)$. On the compact set $\{\ell\leq3R_0\}$, the two transverse metrics are uniformly equivalent. Thus $|Y|$ is bounded throughout $\M$, and compactness of $M$ bounds $p^*(Yf)$. The orthogonal decomposition of an arbitrary unit vector now gives the assertion. These bounds are for fixed geometry and need not remain uniform as $\beta\to0$.
\end{proof}

\subsection{Key estimates}
Recall that $\ell$ is the fiberwise distance from the distinguished section and $V=\ell v$ is the outward radial field. The next proposition distinguishes the growth of $\nabla^{\Ftwo}V$ in the three summands of \eqref{eq:splitting}. This distinction is needed when the leafwise and transverse directions are rescaled.
\begin{proposition}\label{prop:radialestimates}
With respect to the unrescaled metric $g^{T\M}$, the following statements hold.
\begin{enumerate}[label=(\arabic*),leftmargin=2em]
\item The directional derivatives $X\ell$ are uniformly bounded for unit tangent vectors away from $s(M)$, and $X\ell=0$ for $X\in\Fone$.
\item There are constants $C,C'>0$ such that $|\nabla_X^{\Ftwo}V|\leq C\ell+C'$ for vertical unit vectors. For $X\in\Fone$, $\nabla_X^{\Ftwo}V=0$.
\item The length $|\nabla_X^{\Ftwo}V|$ is uniformly bounded for leafwise unit vectors $X\in\LF$.
\end{enumerate}
\end{proposition}
\begin{remark}\label{rem:radiusindependence}
The constants in Proposition~\ref{prop:radialestimates} are independent of the cutoff radius $R_0$. In particular, changing the transverse metric introduces no radius-dependent contribution to these estimates, because both transverse radial derivatives vanish.
\end{remark}
We first prove the vertical estimate by solving the radial Jacobi equation in the symmetric-space fibers.
\begin{lemma}\label{lem:fiberjacobi}
On $G/K=GL^+(k,\R)/SO(k)$, let $\ell$ be distance from the identity coset and let $V=\grad(\ell^2/2)$. There are constants $C,C'>0$ such that $|\nabla_X^{G/K}V|\leq C\ell+C'$ for every unit tangent vector $X$.
\end{lemma}
\begin{proof}
The derivative of $V$ at the endpoint of a radial geodesic is the endpoint derivative of a Jacobi field. The root decomposition of the symmetric space reduces this calculation to scalar equations.

The tangent space at the identity coset $o$ is the space $\mathfrak p$ of symmetric matrices. Every point can be written as $h\exp(H)o$, with $h\in SO(k)$ and $H$ diagonal. Since the left action of $SO(k)$ is isometric and fixes $o$, it suffices to work at $aK$, where $a=\exp(H)$.

The curvature formula \cite[Chapter~IV, Theorem~4.2]{Hel78} is
\begin{equation}\label{eq:symcurvature}
 R_o(X,Y)Z=-[[X,Y],Z].
\end{equation}
If $H=\operatorname{diag}(a_1,\ldots,a_k)$ and $\epsilon_{ij}(H)=a_i-a_j$, then
\begin{equation}\label{eq:rootdecomp}
 \mathfrak p=\mathfrak a\oplus\bigoplus_{i<j}\mathfrak p_{ij},
\end{equation}
where $\mathfrak a$ is diagonal and $\mathfrak p_{ij}$ is spanned by $E_{ij}+E_{ji}$.

For the given tangent vector $X\in T_{aK}(G/K)$, choose a curve $\gamma(s)$ with $\gamma(0)=aK$ and $\dot\gamma(0)=X$. Let $F(s,t)$, $0\leq t\leq1$, be the geodesic from $o$ to $\gamma(s)$, and put $J(t)=\partial_sF(0,t)$. The outward radial convention \eqref{eq:radial} gives
\[
 \partial_tF(s,1)=+V_{\gamma(s)},\qquad
 \nabla_X^{G/K}V=\frac{DJ}{\dd t}(1).
\]
Along $\sigma(t)=\exp(tH)o$, this Jacobi field satisfies
\begin{equation}\label{eq:jacobiequation}
 \frac{D^2J}{\dd t^2}+R(J,\dot\sigma)\dot\sigma=0,
 \qquad J(0)=0,\quad J(1)=X.
\end{equation}
Translate $X$ to $T_o(G/K)$ by the differential of left multiplication by $a^{-1}$, and write $\dd L_{a^{-1}}X=A+\sum_{i<j}X_{ij}$ according to \eqref{eq:rootdecomp}. The diagonal component of the solution is $t\dd L_{\exp(tH)}A$. In a root direction, the solution is
\[
 J_{ij}(t)=\frac{\sinh(t\epsilon_{ij}(H))}{\sinh(\epsilon_{ij}(H))}
             \dd L_{\exp(tH)}X_{ij},
\]
with the limiting value $t\dd L_{\exp(tH)}X_{ij}$ when $\epsilon_{ij}(H)=0$. Therefore
\[
 \dd L_{a^{-1}}(\nabla_X^{G/K}V)
 =A+\sum_{i<j}\epsilon_{ij}(H)\coth(\epsilon_{ij}(H))X_{ij}.
\]
Since $|t\coth t|\leq1+|t|$, while $|\epsilon_{ij}(H)|\leq\sqrt2|H|$ and $|H|=\ell(aK)$, the asserted bound follows. On the fourfold fiber product, the same estimate holds after taking the orthogonal direct sum.
\end{proof}
\begin{proof}[Proof of Proposition~\ref{prop:radialestimates}]
The vertical estimate follows from Lemma~\ref{lem:fiberjacobi}; the horizontal estimates follow by comparing how the two horizontal transports move the distinguished center. We treat the three summands separately.

In a fiber, the triangle inequality gives $|X\ell|\leq|X|$. Since the restriction of $\nabla^{\Ftwo}$ is the fiber Levi--Civita connection, Lemma~\ref{lem:fiberjacobi} also bounds $\nabla_X^{\Ftwo}V$ in vertical directions.

For a transverse lift $X$, Proposition~\ref{prop:transverse} shows that its flow preserves the fiber metric and the distinguished center. It therefore preserves the radial field, giving $X\ell=0$ and $[X,V]=0$. The Koszul identity above then gives $\nabla_X^{\Ftwo}V=[X,V]=0$.

For the leafwise estimate, see also \cite[Lemma~2.1]{Zha17}. To make its uniformity explicit, identify the fibers along a leafwise base curve with a fixed fiber by Bott transport. A point following the Bott lift then becomes a fixed point $z$, while the distinguished center traces a curve $o(t)$ in that fiber. The transported radial field is
\[
 V_t(z)=-\exp_z^{-1}(o(t)).
\]
Each fiber is a Hadamard manifold, and Rauch comparison bounds the differential of $\exp_z^{-1}$ by one. Together with the triangle inequality, this gives
\[
 \left|\frac{\dd}{\dd t}d(z,o(t))\right|\leq|\dot o(t)|,
 \qquad
 \left|\frac{D}{\dd t}V_t(z)\right|\leq|\dot o(t)|.
\]
The speed of $o(t)$ is controlled by the Bott derivative of $g^{TM/F}$ on the compact base. For base curves of leafwise unit speed, this gives a bound independent of $z$ and $R_0$. The bracket description of $\nabla^{\Ftwo}$ identifies the second derivative above with the required leafwise covariant derivative. Finally, decomposing a unit tangent vector into the three orthogonal summands gives the total bound on $X\ell$.
\end{proof}
\begin{remark}\label{rem:normalizedradial}
The product calculation in Section~\ref{sec:cycle} uses the normalized radial field $\widetilde v=\chi_vv$. Its covariant derivative is uniformly bounded in the unrescaled metric. Indeed, on $\{\ell\geq1\}$,
\[
 \nabla_X^{\Ftwo}V=(X\ell)v+\ell\nabla_X^{\Ftwo}v.
\]
Dividing by $\ell$ and applying Proposition~\ref{prop:radialestimates} bounds $\nabla^{\Ftwo}v$ on this region. Compactness handles the region where $\chi_v$ varies. The same conclusion holds for each fixed rescaled metric.
\end{remark}

\subsection{Completeness and the deformed sub-Dirac operator}
The radial derivative bounds provide a smooth proper function with bounded gradient. This proves completeness and provides the proper function used in the essential self-adjointness criterion below. The function $\rho_\infty=\sqrt{1+\ell^2}$ is smooth and proper. Proposition~\ref{prop:radialestimates} gives a constant $L_F$, independent of $R_0$, for which
\[
 |\dd\ell|^2_{g^{T\M}_{\beta,\varepsilon}}
 \leq L_F^2\beta^{-2}+1
 \leq (L_F^2+1)\beta^{-2},\qquad0<\beta,\varepsilon\leq1.
\]
Since $|\dd\rho_\infty|\leq|\dd\ell|$ off the zero section, the same bound holds for $\dd\rho_\infty$ everywhere by continuity.
\begin{corollary}\label{cor:complete}
For every fixed positive choice of the parameters, the metric \eqref{eq:rescaledmetric} is complete. The operators $\Dtotal$, $\Dsub$, and $\Dsub+T\hc(V)$ are essentially self-adjoint on their compactly supported smooth cores.
\end{corollary}
\begin{proof}
A bound on $\dd\rho_\infty$ prevents a finite-length curve from escaping all compact sets: along such a curve $\rho_\infty$ would have bounded variation, contradicting properness. Hence the metric is complete.

For essential self-adjointness, we apply \cite[Proposition~10.2.10]{HR00}: a formally symmetric first-order differential operator is essentially self-adjoint on compactly supported smooth sections if its commutator with a smooth proper function bounded below is bounded. The function $\rho_\infty$ has these properties, and
\[
 [\Dsub+T\hc(V),\rho_\infty]
 =[\Dsub,\rho_\infty]
 =c(\dd\rho_\infty)
\]
is bounded by the preceding gradient estimate. Here the potential commutes with the scalar function $\rho_\infty$. The same criterion applies to $\Dsub$ and $\Dtotal$, whose commutators with $\rho_\infty$ are given by their Clifford actions on $\dd\rho_\infty$.
\end{proof}
We henceforth use the same symbols for the self-adjoint closures. The bounded transforms will be studied through the resolvent formula
\begin{equation}\label{eq:resolventintegral}
 \Fc(D)=\frac2\pi\int_0^\infty D(1+\lambda^2+D^2)^{-1}\,\dd\lambda.
\end{equation}
For the fixed normalizing function $\Fc$, this identity is initially interpreted in the strong operator topology. The differences and commutators below satisfy integrable norm bounds, so their integral formulas converge in norm.

To estimate the deformation term with its parameter dependence visible, choose an adapted $g^{T\M}$-orthonormal frame $e_1,\ldots,e_n$, and put
\begin{equation}\label{eq:scaledframe}
 f_i=\begin{cases}
 \beta^{-1}e_i,&1\leq i\leq m-k,\\
 \varepsilon e_i,&m-k+1\leq i\leq m,\\
 e_i,&m+1\leq i\leq n.
 \end{cases}
\end{equation}
This is an orthonormal frame for \eqref{eq:rescaledmetric}, and
\begin{equation}\label{eq:fullsubdirac}
 \Dsub=\sum_{i=1}^n c_{\beta,\varepsilon}(f_i)
       \nabla^{S(\LF\oplus\Fone)\hct\Lambda^*\Ftwo,\beta,\varepsilon}_{f_i}.
\end{equation}
The connection in this formula includes the mixing term specified above; compare \cite[Equation~(2.17)]{Zha17}.
\begin{remark}\label{rem:theta}
There is a constant $\theta\geq1$, independent of $\beta,\varepsilon\in(0,1]$, $T>0$, and $R_0$, such that
\begin{equation}\label{eq:theta}
 \left\|T\sum_{i=1}^n c_{\beta,\varepsilon}(f_i)
          \hc(\nabla^{\Ftwo}_{f_i}V)\right\|_z
 \leq T\theta(\ell(z)+\beta^{-1}).
\end{equation}
To see the parameter dependence, apply Proposition~\ref{prop:radialestimates} in the frame \eqref{eq:scaledframe}. The leafwise contribution is bounded by a constant times $\beta^{-1}$, the transverse contribution vanishes, and the vertical contribution is bounded by a constant times $1+\ell$. Since $\beta\leq1$, the constant term is absorbed into $\beta^{-1}$. We retain this particular constant $\theta$ in the later positive-curvature estimates; other constants may depend on the indicated fixed data.
\end{remark}

\section{KK-theoretic preliminaries and the deformed cycle}\label{sec:cycle}
The radial potential gives the deformed operator compact resolvent by confining sections near the distinguished section. Its bounded transform then represents the product $b\hct_{C_0(\M)}d_{\M}$ introduced in Section~\ref{sec:setup}. We prove these two statements after recalling the required product criterion and a lemma for the coefficient twist. Set
\[
 A_{\beta,\varepsilon,T}=\Dsub+T\hc(V),\qquad
 H_{\beta,\varepsilon}=L^2(\M,S(\LF\oplus\Fone)\hct\Lambda^*\Ftwo).
\]
When $\beta,\varepsilon$ are fixed, write $A_T$ and $H$.

\subsection{Kasparov-product criteria}
We will identify the product by checking a connection condition and positivity modulo compact operators. The connection condition compares the candidate product operator with the operator on the second factor after tensor multiplication by a vector in the first factor.

Let $[E_1,F_1]\in KK(A,B)$ and $[E_2,F_2]\in KK(B,C)$. Put $E=E_1\hct_BE_2$. For homogeneous $x\in E_1$, let $T_x:E_2\to E$ be $T_xy=x\hct y$, with adjoint $T_x^*(z\hct y)=\ip{x}{z}\,y$. Brackets between homogeneous operators denote graded commutators:
\[
 [P,Q]=PQ-(-1)^{|P||Q|}QP,
\]
where $|P|,|Q|\in\{0,1\}$ are their degrees. Thus $[P,Q]=PQ+QP$ when both operators are odd, and $[P,Q]=PQ-QP$ when either operator is even.
\begin{definition}[{\cite[Definition~18.3.1]{Bla98}}]\label{def:connection}
An adjointable operator $F$ on $E$ is an $F_2$-connection if
\begin{equation}\label{eq:connectioncondition1}
 T_xF_2-(-1)^{|x||F_2|}FT_x
\end{equation}
and
\begin{equation}\label{eq:connectioncondition2}
 F_2T_x^*-(-1)^{|x||F_2|}T_x^*F
\end{equation}
are compact for every homogeneous $x\in E_1$.
\end{definition}
\begin{lemma}[Connes--Skandalis criterion; {\cite[Definition~18.4.1]{Bla98}}]\label{lem:productcriterion}
A Kasparov $(A,C)$-module $(E,F)$ represents $[E_1,F_1]\hct_B[E_2,F_2]$ if $F$ is an $F_2$-connection and
\[
 a[F_1\hct1,F]a^*\geq0\quad\text{modulo compact operators}
\]
for every $a\in A$.
\end{lemma}
\begin{proposition}[{\cite[Propositions~17.2.5 and 17.2.7]{Bla98}}]\label{prop:KKstability}
Two Kasparov modules on the same Hilbert module represent the same class if their operators differ by a compact operator. More generally, if both representing operators are self-adjoint contractions and their norm distance modulo compact operators is less than one, the two modules represent the same class.
\end{proposition}

\subsection{A criterion for vanishing after twisting}
The universal flat bundle will be represented by a projection. The following lemma identifies its product with a compressed cycle and gives a norm condition that makes this compression invertible. We will apply it to the complementary block produced by localization.
\begin{lemma}\label{lem:projectionvanish}
Let $A$ be unital and trivially graded, let $B$ be a $C^*$-algebra, and let $p\in M_N(A)$ be a projection. Set $\mathcal P=pA^N$. If $[H,F]\in KK(A,B)$ with $F^*=F$ and $\norm F\leq1$, then
\[
 [\mathcal P,0]\hct_A[H,F]=[pH^N,pF^{(N)}p].
\]
If moreover
\[
 \norm{[F^{(N)},p]}+\norm{(F^{(N)})^2-1}<1,
\]
this product is zero in $KK(\C,B)$.
\end{lemma}
\begin{proof}
First we verify that compression represents the product. The tensor-product module is $pH^N$; on it, set $G=pF^{(N)}p$. The failure of $G$ to square to the identity is
\begin{equation}\label{eq:compressionidentity}
 G^2-p=p[F^{(N)},p]F^{(N)}p+p\bigl((F^{(N)})^2-1\bigr)p.
\end{equation}
Both terms on the right are compact by the Kasparov conditions for $F$. To check the connection condition, expand the column $T_{p(a_1,\ldots,a_N)}$ and use compactness of $[F,a_j]$ and $[F^{(N)},p]$; the adjoint condition follows by taking adjoints. Since the first factor has operator zero, positivity is automatic. Lemma~\ref{lem:productcriterion} therefore proves the product formula.

Under the norm hypothesis of the lemma, \eqref{eq:compressionidentity} has norm less than one, so $G^2$, and hence $G$, is invertible. We compare $G$ with its sign by functional calculus. With $1$ denoting the identity on $pH^N$,
\[
 1-(G^2)^{-1/2}
 =(G^2-1)(G^2)^{-1/2}\bigl(1+(G^2)^{1/2}\bigr)^{-1}.
\]
The first factor on the right is compact and the others are bounded. Hence $G$ differs compactly from the odd self-adjoint involution $G(G^2)^{-1/2}$. The left algebra of this product is $\C$, so the involution defines a degenerate cycle. This proves that the product class is zero.
\end{proof}

\subsection{Compactness of the deformed cycle}
\begin{proposition}\label{prop:deformedcycle}
For every $T>0$ and $\beta,\varepsilon\in(0,1]$, \eqref{eq:deformedclass} is a Kasparov $(C(M),\C)$-module. The operator $A_T$ has compact resolvent, and
\[
 \Dom A_T=\Dom\Dsub\cap\{u\in H:\ell u\in H\}.
\]
For fixed parameters its graph norm is equivalent to
$\norm u+\norm{\Dsub u}+T\norm{\ell u}$.
\end{proposition}
\begin{proof}
We first show that the graph norm controls both $\Dsub u$ and $\ell u$. The first bound gives compactness on compact sets; the second prevents $L^2$ mass from escaping to infinity. Together they prove compact resolvent. We then verify the bounded-transform commutators.

On the smooth compactly supported core, Clifford anticommutation and compatibility of $\nabla^{\Ecal}$ give
\begin{equation}\label{eq:square}
 A_T^2=\Dsub^2+T^2\ell^2
 +T\sum_{i=1}^nc_{\beta,\varepsilon}(f_i)\hc(\nabla^{\Ftwo}_{f_i}V).
\end{equation}
In particular, \eqref{eq:theta} yields
\begin{align}
 \norm{A_Tu}^2
 &\geq\norm{\Dsub u}^2+T^2\norm{\ell u}^2
       -T\theta\int_{\M}(\ell+\beta^{-1})|u|^2\dd\vol\nonumber\\
 &\geq\norm{\Dsub u}^2+\frac{T^2}{2}\norm{\ell u}^2
       -\left(T\theta\beta^{-1}+\frac{\theta^2}{2}\right)\norm u^2.
 \label{eq:confinement}
\end{align}
The second line uses $T\theta\ell\leq T^2\ell^2/2+\theta^2/2$. If a sequence of core sections converges in the graph norm of $A_T$, this estimate makes it Cauchy in the graph norms of $\Dsub$ and multiplication by $\ell$. Since these operators are closed, $\Dom A_T$ is contained in the asserted intersection. Conversely, if both $\Dsub u$ and $\ell u$ lie in $H$, then the distribution $A_Tu$ lies in $H$. Essential self-adjointness identifies the maximal distributional domain with the closure of the core, giving the reverse inclusion. The triangle inequality gives the upper graph estimate, and \eqref{eq:confinement} gives the lower one for fixed parameters.

To prove that the inclusion $\Dom A_T\hookrightarrow L^2(\M,\Ecal)$ is compact when $\Dom A_T$ carries the graph norm, it suffices to show that every graph-norm-bounded sequence has a subsequence converging in $L^2(\M,\Ecal)$. Fix $T,\beta,\varepsilon$ and let $(u_j)\subset\Dom A_T$ be any such bounded sequence. Equation~\eqref{eq:confinement} bounds both $\Dsub u_j$ and $\ell u_j$ uniformly in $j$. Fix a compact set $K\subset\M$ and choose compactly supported smooth functions $\chi,\widetilde\chi$ taking values in $[0,1]$, with $\chi=1$ near $K$ and $\widetilde\chi=1$ near $\supp\chi$. Since $\Dsub$ is elliptic, the cutoff estimate in Proposition~\ref{prop:garding} gives
\[
 \norm{\chi u_j}_{H^1}
 \leq C_{\chi,\widetilde\chi}
 \bigl(\norm{\widetilde\chi\Dsub u_j}_{L^2}
       +\norm{\widetilde\chi u_j}_{L^2}\bigr).
\]
The right-hand side is uniformly bounded. All sections $\chi u_j$ are supported in the fixed compact set $\supp\chi$, so Proposition~\ref{prop:rellich} gives a subsequence $(\chi u_{j_k})$ converging in $L^2(\M,\Ecal)$. Because $\chi=1$ on $K$, the subsequence $(u_{j_k})$ converges in $L^2(K,\Ecal)$.

To pass from local to global convergence, we use the uniform bound on $\ell u_j$. For every $L>0$,
\begin{equation}\label{eq:tailcontrol}
 \int_{\{\ell\geq L\}}|u_j|^2\dd\vol\leq L^{-2}\norm{\ell u_j}^2.
\end{equation}
The sets $\{\ell\leq L\}$, $L=1,2,\ldots$, are compact and cover $\M$. Successive subsequence extraction and a diagonal choice therefore give a single subsequence converging in $L^2$ on each of these sets. This subsequence is Cauchy in global $L^2$: first choose $L$ large enough that \eqref{eq:tailcontrol} makes the tails uniformly small, and then use convergence on $\{\ell\leq L\}$ to control the difference there. Completeness of $L^2$ gives global convergence. Thus the graph-domain inclusion is compact, or equivalently $(1+A_T^2)^{-1}$ is compact. This argument is a confining-potential criterion of the type in \cite[Proposition~B.1]{LS19}.

It remains to check commutators with the $C(M)$-action. For $f\in C^\infty(M)$,
\begin{equation}\label{eq:ATcommutator}
 [A_T,f]=[\Dsub,f]=\sum_{i=1}^n c_{\beta,\varepsilon}(f_i)f_i(p^*f)
\end{equation}
is bounded by Lemma~\ref{lem:basecommutators}. Put $R_\lambda=(1+\lambda^2+A_T^2)^{-1}$. Resolvent algebra gives
\begin{equation}\label{eq:boundedcommutator}
 [f,\Fc(A_T)]=\frac2\pi\int_0^\infty R_\lambda
 \bigl((1+\lambda^2)[f,A_T]+A_T[A_T,f]A_T\bigr)
 R_\lambda\,\dd\lambda.
\end{equation}
The resolvents and $A_TR_\lambda$ are compact, and the integrand norm is bounded by a constant times $(1+\lambda^2)^{-1}$. The integral therefore converges in norm to a compact operator. Uniform approximation by smooth functions extends this conclusion to every $f\in C(M)$. Self-adjointness and compactness of $1-\Fc(A_T)^2=(1+A_T^2)^{-1}$ give the remaining Kasparov conditions.
\end{proof}

\subsection{Identification of the Kasparov product}
\begin{proposition}\label{prop:product}
For every $T>0$ and $\beta,\varepsilon\in(0,1]$,
\[
 b\hct_{C_0(\M)}d_{\M}
 =[H_{\beta,\varepsilon},\Fc(A_{\beta,\varepsilon,T})].
\]
\end{proposition}
\begin{proof}
We apply Lemma~\ref{lem:productcriterion} to the cycle established in Proposition~\ref{prop:deformedcycle}. After identifying the tensor-product module, the connection condition will follow from the tensor-product connection, and positivity will follow from the nonnegative radial term in the anticommutator with the Bott operator.

The splitting $S(T\M)\simeq S(\LF\oplus\Fone)\hct S(\Ftwo)$, the graded flip, and the Clifford-symbol identification give the unitary
\[
 C_0(\M,S^*(\Ftwo))\hct_{C_0(\M)}L^2(\M,S(T\M))
 \simeq L^2(\M,\Ecal).
\]
Fiberwise it is obtained from
\[
 S^*(\Ftwo)\hct S(T\M)
 \longrightarrow S(T\M)\hct S^*(\Ftwo)
 \longrightarrow S(\LF\oplus\Fone)\hct\Lambda^*\Ftwo,
\]
where the first arrow sends $x\hct y$ to $(-1)^{|x||y|}y\hct x$, and the endomorphism convention is $s\hct\varphi\mapsto(u\mapsto s\varphi(u))$. These identifications preserve the induced connections. With \eqref{eq:dualaction},
\begin{equation}\label{eq:positivesign}
 -\hc^*_{\mathrm{spin}}(\widetilde v)\hct1
 \longleftrightarrow J:=+\hc(\widetilde v).
\end{equation}
Thus the operator on the first factor is represented by $J$ with the sign shown in \eqref{eq:positivesign}.

\emph{The connection condition.}
For a homogeneous $x\in C_c^\infty(\M,S^*(\Ftwo))$, let $T_x$ denote tensor multiplication, viewed through this unitary as a map from $L^2(\M,S(T\M))$ to $L^2(\M,\Ecal)$. Put $D_x=(-1)^{|x|}\Dtotal$ and
\[
 W_x=A_TT_x-T_xD_x.
\]
The tensor-product formula for $\nabla^{\Ecal}$ gives cancellation of the derivatives acting on the second factor. The remainder consists of Clifford multiplication by $\nabla^{S^*(\Ftwo)}x$ and the potential term $T\hc(V)T_x$. Both have compact support because $x$ does, so $W_x$ is bounded.

With $a^2=1+\lambda^2$, let $R_A=(a^2+A_T^2)^{-1}$ and $R_x=(a^2+D_x^2)^{-1}$. Then
\begin{equation}\label{eq:connectionintegral}
 \Fc(A_T)T_x-T_x\Fc(D_x)
 =\frac2\pi\int_0^\infty R_A
 \bigl(a^2W_x-A_TW_xD_x\bigr)R_x\,\dd\lambda.
\end{equation}
Here the products containing the unbounded operators are understood through the bounded factors $A_TR_A$ and $D_xR_x$. The integrand is compact because $R_A$ and $A_TR_A$ are compact, and its norm is at most $5\norm{W_x}/(4a^2)$. Thus the integral converges in norm to a compact operator. Since $\Fc(D_x)=(-1)^{|x|}\Fc(\Dtotal)$, this proves \eqref{eq:connectioncondition1}; taking adjoints proves \eqref{eq:connectioncondition2}. Density of compactly supported smooth sections in $C_0(\M,S^*(\Ftwo))$ extends both conditions to every $x$.

\emph{The positivity condition.}
Since $J$ and $A_T$ are odd, their graded commutator is $[J,A_T]=JA_T+A_TJ$. We separate a bounded term from the positive contribution of the radial potential. By Remark~\ref{rem:normalizedradial}, $K=[J,\Dsub]$ is bounded for fixed $\beta,\varepsilon$. The potential gives
\begin{equation}\label{eq:positivepotential}
 [J,A_T]=K+W,
 \qquad W=2T\chi_v\ell\geq0.
\end{equation}
Applying resolvent algebra first with a finite upper integration limit gives
\begin{equation}\label{eq:positivityintegral}
 [J,\Fc(A_T)]
 =\frac2\pi\int_0^\infty R_A
 \bigl(a^2(K+W)+A_T(K+W)A_T\bigr)R_A\,\dd\lambda.
\end{equation}
The terms containing $W$ are interpreted as quadratic forms after the displayed resolvents; \eqref{eq:confinement} makes these forms bounded for each $\lambda$. The $K$ contribution is compact and norm-integrable, with the same $O(a^{-2})$ bound as above.

Each truncated $W$ contribution is positive, and these operators increase with the upper integration limit. They are also uniformly bounded. Indeed, the truncated normalizing functions in \eqref{eq:resolventintegral} have norm at most one, so their anticommutators with $J$ have norm at most $2\norm J$; the $K$ contributions have a uniform norm bound as well. The $W$ contributions therefore converge strongly to a positive operator. Formula~\eqref{eq:positivityintegral} is consequently the sum of a positive operator and a compact operator. Since $J$ commutes with base functions, the required positivity remains true after multiplication on the two sides by $f$ and $f^*$ for every $f\in C(M)$. Lemma~\ref{lem:productcriterion} completes the proof.
\end{proof}
\section{Bismut--Lebeau analytic localization}\label{sec:localization}
Fix $R_0>1$ and $\beta,\varepsilon\in(0,1]$. We identify the class of $A_T$ by the Gaussian localization method of \cite{BL91}, without a positivity assumption. The first step replaces $A_T$, through a compactly supported homotopy, by an operator whose exact tubular model is a base Dirac operator plus a vertical harmonic oscillator. Compression to the Gaussian ground state then recovers the base Dirac operator. The remaining states have a spectral gap that grows with $T$.

The normalization accounts for the possible nonorthogonality of $Ts(M)$ and $\Ftwo$, including its effects on the measure and connection. Each homotopy is constructed with $T>0$ fixed; the subsequent block estimates use large $T$. Unless stated otherwise, constants may depend on the fixed geometry, $R_0,\beta,\varepsilon$, and the tubular radius, but are independent of $T$.

\subsection{Tubular coordinates and bundle identifications}
The fiberwise exponential map $\Phi:N\to\M$ in \eqref{eq:fiberexp} lets us perform localization on the vector bundle $\pi:N=\Ftwo|_{s(M)}\to M$, with $p\circ\Phi=\pi$. Equip each fiber $N_y$ with the Euclidean metric $g^N_y$ induced by $g^{\Ftwo}$ at $s(y)$. Let $g_0^M$ be the metric on $TM$ for which
\[
 p_*:(\LF\oplus\Fone)|_{s(M)}\longrightarrow TM
\]
is an isometry for the rescaled metric \eqref{eq:rescaledmetric}. This defines $g_0^M$ through the horizontal bundle; it can differ from the metric induced on $s(M)$. Put $S_M=S(TM,g_0^M)$ and fix a metric connection $\nabla^N$ on $N$, with the chosen spin structure. One possible choice is the restriction of the projected vertical connection along $Ts(M)$.

We use the product measure associated with $g_0^M$ and the Euclidean fibers so that a normalized fiberwise Gaussian will give an isometric embedding of base spinors. Set
\begin{equation}\label{eq:referenceHilbert}
 \mathcal S_0=\pi^*S_M\hct\Lambda^*\pi^*N,
 \qquad
 \Hcal=L^2(N,\mathcal S_0,\dd\mu_0),
 \qquad
 \dd\mu_0=\dd\vol_{g_0^M}(y)\,\dd Z_y.
\end{equation}
Here $\dd Z_y$ is Euclidean volume on $N_y$. At $z=(y,Z)$, the connection $\nabla^N$ lifts $X\in T_yM$ to the tangent vector $X^H_0(z)\in T_zN$ obtained by parallel-transporting $Z$ along a curve with initial velocity $X$. In particular, $\dd\pi_z(X^H_0(z))=X$. Declare this horizontal bundle orthogonal to $\ker\dd\pi$, use $g_0^M$ on it and $g^N$ on the vertical tangent spaces, and denote the resulting Riemannian metric on $TN$ by $g_{\mathrm{mod}}$. Its Riemannian volume is $\dd\mu_0$.

Let $\nabla^0$ be the product connection on $\mathcal S_0$ induced by the Levi--Civita spin connection of $g_0^M$ and the exterior-algebra connection of $\nabla^N$. It defines the horizontal and vertical model operators
\begin{equation}\label{eq:modeloperators}
 D^H=\sum_{i=1}^m c_M(e_i)\nabla^0_{e_{i,0}^H},
 \qquad
 D^N=\sum_{a=1}^d c_N(u_a)\partial_{Z_a},
 \qquad D_{\mathrm{mod}}=D^H+D^N.
\end{equation}
Here $e_i$ and $u_a$ are local orthonormal frames of $(TM,g_0^M)$ and $(N,g^N)$, respectively, and $e_{i,0}^H$ denotes the lift $X^H_0$ for $X=e_i$. Vertical Clifford actions include the grading of $S_M$; equivalently $c_N(u)=1\hct(\epsilon(u)-\iota(u))$ and $\hc(u)=1\hct(\epsilon(u)+\iota(u))$.

In an orthonormal local frame of $N$ with connection matrix $\omega$,
\begin{equation}\label{eq:modelhorizontal}
 X^H_0=X-\sum_{a,b=1}^d\omega_{ab}(X)Z_b\partial_{Z_a}.
\end{equation}
Skew-symmetry of $\omega$ gives $X^H_0|Z|=0$ and invariance of Euclidean fiber volume under parallel transport.
\begin{lemma}\label{lem:modelsymmetric}
The operators $D^H$ and $D^N$ are odd and formally symmetric on $\Hcal$, and their sum $D_{\mathrm{mod}}$ is a first-order elliptic operator. For $z=(y,Z)\in N$, decompose $\xi\in T_z^*N$ into $\xi_H\in T_y^*M$ and $\xi_N\in N_y^*$ using the horizontal distribution defined by $\nabla^N$. With the symbol convention $\sigma_D(\dd f)=[D,f]$, its principal symbol is
\[
 \sigma_{D_{\mathrm{mod}}}(z,\xi)
 =c_M(\xi_H)+c_N(\xi_N),
\]
where covectors are identified with vectors using $g_0^M$ and $g^N$, respectively.
\end{lemma}
\begin{proof}
We check symmetry with respect to the product measure in \eqref{eq:referenceHilbert}. For $D^N$, this is Euclidean integration by parts. In \eqref{eq:modelhorizontal}, the vertical linear vector field has divergence $-\Tr\omega(X)=0$. Thus
\[
 \operatorname{div}_{\mu_0}(X^H_0)=\operatorname{div}_{g_0^M}X.
\]
The usual symmetry calculation for the spin Dirac operator on $M$ consequently applies to $D^H$ with the additional metric connection on the exterior-algebra factor. The symbol formula follows directly from \eqref{eq:modeloperators}. The graded Clifford relations give
\[
 \sigma_{D_{\mathrm{mod}}}(z,\xi)^2
 =-\bigl(|\xi_H|_{g_0^M}^2+|\xi_N|_{g^N}^2\bigr)\operatorname{Id}
 =-|\xi|_{g_{\mathrm{mod}}}^2\operatorname{Id}.
\]
Thus the symbol is invertible for every nonzero $\xi$, proving ellipticity.
\end{proof}
The full geometric sub-Dirac operator of $g_{\mathrm{mod}}$ has the same principal symbol as $D_{\mathrm{mod}}$. Their zeroth-order difference will be the final correction in the normalization below; this correction is what gives an exact model, including its connection terms.

\subsection{Compactly supported normalization}
Fix $r_*>0$, independent of $T$. To carry out Gaussian localization, we will embed spinors on $M$ into sections on $N$ by multiplying their pullbacks by a normalized, cutoff Gaussian in the fiber variable. This embedding, constructed in \eqref{eq:gaussianembedding} below, will be supported in $\{|Z|\leq2r_*\}$. We therefore construct an auxiliary operator equal to
\[
 D^H+D^N+T\hc(Z)
\]
on the larger region $\{|Z|<4r_*\}$. Outside a compact set it will agree with $A_T$ after a fixed bundle and density identification. The next proposition constructs this representative and proves that it has the same $K$-homology class.
\begin{proposition}\label{prop:normalization}
There is a $C(M)$-linear unitary $\mathcal U_0:H_{\beta,\varepsilon}\to\Hcal$ and a smooth family of odd, formally symmetric elliptic operators
\begin{equation}\label{eq:localizingfamily}
 \Acal_{T,t}=\Dcal_t+T\hc(Z),\qquad0\leq t\leq1,
\end{equation}
with the following properties.
\begin{enumerate}[label=(\roman*),leftmargin=2em]
\item $\Acal_{T,0}=\mathcal U_0A_T\mathcal U_0^{-1}$; all $\Dcal_t$ agree with $\Dcal_0$ outside a fixed compact set.
\item The endpoint $\Acal_T:=\Acal_{T,1}$ equals $D_{\mathrm{mod}}+T\hc(Z)$ on $\{|Z|<4r_*\}$.
\item For each fixed $T>0$, the operators $\Acal_{T,t}$, initially defined on $C_c^\infty(N,\mathcal S_0)\subset\Hcal$, have self-adjoint closures with compact resolvent. These closures have a common domain, and their graph norms are uniformly equivalent for $t\in[0,1]$.
\item Their bounded transforms define a Kasparov homotopy. In particular,
\begin{equation}\label{eq:normalizationclass}
 [H_{\beta,\varepsilon},\Fc(A_T)]=[\Hcal,\Fc(\Acal_T)]
 \quad\text{in }KK(C(M),\C).
\end{equation}
\end{enumerate}
\end{proposition}
\begin{proof}
We first deform the metric to $g_{\mathrm{mod}}$ near the zero section while preserving the radial potential $T\hc(Z)$. We then transfer the operators to the fixed Hilbert space $\Hcal$ and correct the remaining zeroth-order difference. The final two steps establish the common domains and the Kasparov homotopy.

\emph{Step 1: metric interpolation preserving the radial vector.}
Write $g_0=\Phi^*g^{T\M}_{\beta,\varepsilon}$ for the pulled-back metric on $TN$. At $z=(y,Z)$, its horizontal space is the $g_0$-orthogonal complement of $\ker\dd\pi_z$. For $X\in T_yM$, define $X^H(z)$ to be the unique vector in that complement with $\dd\pi_z(X^H(z))=X$. Thus $X^H$ is the lift defined by $g_0$, whereas $X^H_0$ is defined by the vector-bundle connection $\nabla^N$.

Both lifts project to $X$, so their difference is vertical. Under the canonical identification $\ker\dd\pi_z\simeq N_y$, there is a unique smooth linear map $L_z:T_yM\to N_y$ such that
\[
 X^H=X^H_0+L_zX.
\]
Let $h_z$ be the positive form on $T_yM$ given by $h_z(X,Y)=g_0(X^H,Y^H)$. For vertical vectors $U,W\in N_y\simeq\ker\dd\pi_z$, define the positive self-adjoint endomorphism $G_z$ by
\[
 g_0(U,W)=g^N_y(G_zU,W).
\]
The Gauss lemma for the fiberwise exponential map gives
\begin{equation}\label{eq:gaussradial}
 G_zZ=Z.
\end{equation}
Indeed, $g_0(Z,W)=g^N_y(Z,W)$ for all vertical coordinate vectors $W$. Here $Z\in N_y$ also denotes the radial vertical vector at $(y,Z)$, as in \eqref{eq:radial}.

Choose a smooth cutoff $\chi_g(|Z|)\in[0,1]$, equal to one on $|Z|\leq5r_*$ and zero on $|Z|\geq6r_*$. The following interpolation gives a metric \(g_t\) that equals \(g_{\mathrm{mod}}\) at \(t=1\) on \(\{|Z|\leq5r_*\}\) and remains equal to \(g_0\) on \(\{|Z|\geq6r_*\}\) for every \(t\in[0,1]\).
\begin{align}
 L_{t,z}&=(1-t\chi_g)L_z,\nonumber\\
 h_{t,z}&=(1-t\chi_g)h_z+t\chi_g g_0^M,\label{eq:metricinterpolation}\\
 G_{t,z}&=(1-t\chi_g)G_z+t\chi_g I.\nonumber
\end{align}
For $X\in T_yM$, put $X^{H,t}=X^H_0+L_{t,z}X$. Define $g_t$ at $z=(y,Z)$ by
\[
 g_t\bigl(X^{H,t}+U,Y^{H,t}+W\bigr)
 =h_{t,z}(X,Y)+g^N_y(G_{t,z}U,W),
\]
where $X,Y\in T_yM$ and $U,W\in N_y\simeq\ker\dd\pi_z$. Since \eqref{eq:metricinterpolation} takes convex combinations of positive forms, it defines a smooth positive metric $g_t$ on $TN$. Each $g_t$ is complete because it is a compactly supported change of the complete metric $g_0$.

By \eqref{eq:gaussradial}, $G_{t,z}Z=Z$. Therefore \(G_{t,z}^{-1/2}\) is an isometry from \((N_y,g_y^N)\) to \((\ker d\pi_z,g_t)\) under the canonical identification \(\ker d\pi_z\simeq N_y\), and it fixes the radial vector \(Z\). The horizontal bundle defined by \(\nabla^N\) can be isometrically identified with the $g_t$-horizontal bundle. Together with the vertical isometries and the compatible spin structures discussed below, they give the graded unitary identification
\[
 S\bigl((\ker\dd\pi)^{\perp_{g_t}},g_t\bigr)
 \hct\Lambda^*(\ker\dd\pi)
 \xrightarrow{\;\mathcal J_t\;}\pi^*S_M\hct\Lambda^*\pi^*N=\mathcal S_0.
\]
Equip the \(g_t\)-horizontal bundle with the spin structure transported from \(\pi^*TM\) by the horizontal isometry above. This isometry then induces the corresponding unitary identification of spinor bundles. Since the vertical isometry fixes $Z$, the identification $\mathcal J_t$ satisfies
\begin{equation}\label{eq:potentialpreserved}
 \mathcal J_t\hc_{g_t}(Z)\mathcal J_t^{-1}=\hc(Z).
\end{equation}

\emph{Step 2: the density and zeroth-order corrections.}
Relative to the model horizontal--vertical splitting, the map sending $X^H_0+U$ to $X^{H,t}+U$ is represented by
\[
 \begin{pmatrix} I&0\\ L_{t,z}&I \end{pmatrix}.
\]
Its determinant is $1$, so this change of frame contributes no additional factor to the volume density. The metric in the new frame has horizontal block $h_{t,z}$ and vertical block $G_{t,z}$, and hence
\[
 \dd\vol_{g_t}=j_t(y,Z)\dd\mu_0,
 \qquad
 j_t=\bigl(\det_{g_0^M}h_t\,\det G_t\bigr)^{1/2}.
\]
Here, at $z=(y,Z)$,
\[
 \det_{g_0^M}h_t=\det\bigl(h_{t,z}(e_i,e_j)\bigr)_{i,j=1}^m,
\]
where $e_1,\ldots,e_m$ is any $g_0^M$-orthonormal basis of $T_yM$; the value is independent of this choice. The induced unitary between the $L^2$ spaces is
\begin{equation}\label{eq:densityunitary}
 \mathcal U_t=j_t^{1/2}\mathcal J_t:
 L^2\!\left(N,S\bigl((\ker\dd\pi)^{\perp_{g_t}},g_t\bigr)
 \hct\Lambda^*(\ker\dd\pi),\dd\vol_{g_t}\right)
 \longrightarrow\Hcal=L^2(N,\mathcal S_0,\dd\mu_0).
\end{equation}
The source bundle carries the fiber metrics induced by $g_t$. At $t=0$ we include pullback by $\Phi$ to obtain the unitary from $H_{\beta,\varepsilon}$. All these maps commute with base multiplication. 

Let $\widetilde D_t$ be the full geometric sub-Dirac operator for $g_t$, conjugated by $\mathcal U_t$ to act on $\Hcal$. On $|Z|<5r_*$, $D_{\mathrm{mod}}-\widetilde D_1$ is a smooth odd zeroth-order operator. It is pointwise self-adjoint, because both first-order operators have the same principal symbol and are formally symmetric for $\dd\mu_0$. Choose $\chi_c$ equal to one on $|Z|\leq4r_*$ and supported in $|Z|<5r_*$. Define the globally extended multiplication operator
\begin{equation}\label{eq:zerothnormalization}
 C=\chi_c(D_{\mathrm{mod}}-\widetilde D_1),
 \qquad \Dcal_t=\widetilde D_t+tC.
\end{equation}
At $t=1$ it cancels the full difference between $\widetilde D_1$ and $D_{\mathrm{mod}}$ on $|Z|<4r_*$. Together with \eqref{eq:potentialpreserved}, this proves (i) and (ii).

\emph{Step 3: self-adjointness and common domains.}
To prove essential self-adjointness, we apply the criterion of \cite[Proposition~10.2.10]{HR00} used in Corollary~\ref{cor:complete}, now with the smooth proper function $\rho_\infty\circ\Phi$ on $N$, which is bounded below. Outside a fixed compact set, $\Dcal_t$ agrees with $\Dcal_0$, so $[\Dcal_t,\rho_\infty\circ\Phi]$ is bounded there by the commutator estimate in that corollary, transported by the fixed unitary identification. On the compact set, the principal symbols vary smoothly with $t\in[0,1]$, so these commutators are uniformly bounded there as well. Since the potential $T\hc(Z)$ commutes with $\rho_\infty\circ\Phi$, the criterion proves essential self-adjointness of both $\Acal_{T,t}$ and $\Dcal_t$ on $C_c^\infty(N,\mathcal S_0)$.

The principal symbol of every $\Dcal_t$ anticommutes with $\hc(Z)$. Hence
\begin{equation}\label{eq:auxsquare}
 \Acal_{T,t}^2=\Dcal_t^2+T^2|Z|^2+T C_t,
 \qquad C_t=[\Dcal_t,\hc(Z)]
\end{equation}
has a zeroth-order last term. Write $C_t(y,Z)$ for the fiber endomorphism defining the multiplication operator $C_t$ at $(y,Z)\in N$. On the fixed compact deformation region, $\norm{C_t(y,Z)}$ is bounded uniformly for $t\in[0,1]$, by smooth dependence on $(t,y,Z)$ and compactness. Outside that region, $\Dcal_t$ agrees with $\Dcal_0$, and \eqref{eq:theta}, after the fixed unitary identification, gives a bound growing at most linearly in $|Z|$. Combining these two bounds gives a constant $C_*>0$ such that
\begin{equation}\label{eq:auxanticommutator}
 \norm{C_t(y,Z)}\leq C_*(1+|Z|),\qquad y\in M,\quad Z\in N_y,\quad t\in[0,1].
\end{equation}
The constant \(C_*\) may depend on \(R_0,\beta,\varepsilon,r_*\) and the fixed geometric data and cutoff functions, but is independent of \(T\) and \(t\in[0,1]\).

Combining \eqref{eq:auxsquare} and \eqref{eq:auxanticommutator} gives, on $C_c^\infty(N,\mathcal S_0)$,
\begin{equation}\label{eq:auxconfinement}
 \norm{\Acal_{T,t}u}^2
 \geq\norm{\Dcal_tu}^2+\frac{T^2}{2}\norm{|Z|u}^2
       -\left(TC_*+\frac{C_*^2}{2}\right)\norm u^2.
\end{equation}
Estimate~\eqref{eq:auxconfinement} controls both $\Dcal_tu$ and $|Z|u$ by the graph norm. Local elliptic regularity and the Rellich Lemma give compactness on each compact set, while \eqref{eq:tailcontrol} controls the $L^2$ mass outside that set. As in Proposition~\ref{prop:deformedcycle}, this proves compact resolvent.

Applying the argument in Proposition~\(\ref{prop:deformedcycle}\) to \(\eqref{eq:auxconfinement}\), we obtain
\[
\operatorname{Dom}\mathcal A_{T,t}
=\operatorname{Dom}\mathcal D_t\cap\operatorname{Dom}|Z|.
\]

To prove that \(\operatorname{Dom}\mathcal A_{T,t}\) is independent of \(t\), choose a compactly supported smooth function \(\eta\) equal to one on a sufficiently large compact set whose interior contains the supports of \(\mathcal D_t-\mathcal D_s\) for all \(s,t\in[0,1]\). Since \(\mathcal A_{T,t}\) is elliptic and its coefficients depend smoothly on \(t\in[0,1]\), the local elliptic estimate on this fixed compact set holds with a constant independent of \(t\). We have
\[
 \norm{\eta u}_{H^1}\leq C_T(\norm{\Acal_{T,t}u}+\norm u),
 \qquad0\leq t\leq1.
\]
Smooth dependence of the compactly supported first-order coefficients then gives
\begin{equation}\label{eq:graphcontinuity}
 \norm{(\Acal_{T,t}-\Acal_{T,s})u}
 \leq C_T|t-s|(\norm{\Acal_{T,s}u}+\norm u).
\end{equation}
Since $|t-s|\leq1$, the triangle inequality bounds the $t$-graph norm by the $s$-graph norm uniformly in $s,t$; exchanging $s$ and $t$ gives the reverse bound. The graph norms are therefore uniformly equivalent on the common core. Their completions have the same image in $\Hcal$, which is the common closure domain. This proves (iii).

\emph{Step 4: the Kasparov homotopy.}
We assemble the bounded transforms into an operator on $C([0,1],\Hcal)$. This requires strong continuity of the transforms and their adjoints, together with norm continuity of the compact defects and commutators.

Since $\Acal_{T,t}$ and $\Acal_{T,s}$ have the same domain, their resolvents satisfy
\begin{equation}\label{eq:normalizationresolvent}
 (\Acal_{T,t}\pm i)^{-1}-(\Acal_{T,s}\pm i)^{-1}
 =(\Acal_{T,t}\pm i)^{-1}
 (\Acal_{T,s}-\Acal_{T,t})(\Acal_{T,s}\pm i)^{-1}.
\end{equation}
Applying \eqref{eq:graphcontinuity} to $(\Acal_{T,s}\pm i)^{-1}u$ in the common domain, and using
$\norm{(\Acal_{T,s}\pm i)^{-1}}\leq1$ and
$\norm{\Acal_{T,s}(\Acal_{T,s}\pm i)^{-1}}\leq1$,
\eqref{eq:normalizationresolvent} gives
\[
 \norm{(\Acal_{T,t}\pm i)^{-1}-(\Acal_{T,s}\pm i)^{-1}}
 \leq2C_T|t-s|.
\]
Thus $(\Acal_{T,t}\pm i)^{-1}$ depends continuously on $t$ in operator norm. Functional calculus then gives norm continuity of $h(\Acal_{T,t})$ for every $h\in C_0(\R)$. The norm continuity of  resolvent also implies the strong continuity of $t\mapsto\Fc(\Acal_{T,t})$. 

Let $I=[0,1]$ and define, pointwise on $C(I,\Hcal)$,
\[
 (\mathfrak F_Tu)(t)=\Fc(\Acal_{T,t})u(t).
\]
The continuity just proved makes $\mathfrak F_T$ an adjointable odd self-adjoint operator.

The function $t\mapsto1-\Fc(\Acal_{T,t})^2$ is norm-continuous with values in $\K(\Hcal)$. Outside the fixed compact deformation region, Lemma~\(\ref{lem:basecommutators}\) bounds \([A_T,f]\). On that compact region, the coefficients of \([\mathcal A_{T,t},f]\) depend smoothly on \(t\in[0,1]\). Therefore, for each \(f\in C^\infty(M)\), the commutators \([\mathcal A_{T,t},f]\) depend continuously on \(t\) in operator norm and are uniformly bounded. Formula~\eqref{eq:boundedcommutator} proves that
\[
 t\longmapsto[\Fc(\Acal_{T,t}),f]
\]
is also norm-continuous with compact values. Uniform approximation of $f$ extends this to all $f\in C(M)$. Under the identification $\K_{C(I)}(C(I,\Hcal))=C(I,\K(\Hcal))$, these fields are precisely the compact operators required for a Kasparov cycle. Therefore
\[
 [C(I,\Hcal),\mathfrak F_T]\in KK(C(M),C(I))
\]
is a Kasparov homotopy with the stated endpoints. This proves (iv).
\end{proof}


\subsection{Expansion and rescaling of the normalized operator}
The outcome of Proposition~\ref{prop:normalization} is the exact normal form
\begin{equation}\label{eq:exactnormalform}
 \Acal_T=D^H+D^N+T\hc(Z),\qquad |Z|<4r_*.
\end{equation}
This separates the horizontal Dirac operator from the vertical operator whose Gaussian ground state will define the low-energy space. Define the unitary dilation on the fiberwise Euclidean $L^2$ space by
\[
 (\delta_Tu)(y,Z)=T^{-d/4}u(y,Z/\sqrt T).
\]
The horizontal lifts defined by $\nabla^N$ are linear in $Z$, as in \eqref{eq:modelhorizontal}, and their parallel transport preserves the fiber metric. Consequently,
\begin{equation}\label{eq:unitaryrescaling}
 \delta_TD^H\delta_T^{-1}=D^H,
 \qquad
 \delta_T(D^N+T\hc(Z))\delta_T^{-1}
 =\sqrt T(D^N+\hc(Z)).
\end{equation}
The factor $T^{-d/4}$ accounts for the change in Euclidean fiber volume. We use \eqref{eq:exactnormalform} on sections supported in $|Z|<4r_*$; the analysis outside that region uses the confinement estimate \eqref{eq:auxconfinement}.

\subsection{The vertical harmonic oscillator}
The square of the vertical model has a gap of order $T$ above its Gaussian kernel. Let $N_\Lambda$ denote degree on the vertical exterior-algebra factor. A direct computation gives
\begin{equation}\label{eq:oscillatorsquare}
 Q_T:=D^N+T\hc(Z),\qquad
 Q_T^2=-\Delta_Z+T^2|Z|^2+T(2N_\Lambda-d).
\end{equation}
We use the following kernel and spectral-gap statements from \cite[Section~4.5]{Zha01}; see also the localization theory of \cite{BL91}.
\begin{proposition}\label{prop:oscillator}
On the vertical factor over each $y$, the kernel of $Q_T$ is the one-dimensional even space generated by
\[
 g_T(Z)=(T/\pi)^{d/4}e^{-T|Z|^2/2}\otimes1.
\]
There is a constant $c_{\mathrm{osc}}>0$, independent of $y$ and $T$, such that every nonzero eigenvalue of $Q_T^2$ is at least $c_{\mathrm{osc}}T$. On the full spinor-valued fiber, the kernel is $S_{M,y}\hct\C g_T$.
\end{proposition}

\subsection{The Gaussian embedding and low-energy projection}
We cut off the oscillator ground state so that its image is supported in the exact-model region. Choose a smooth radial cutoff $0\leq\rho\leq1$ with
\[
 \rho(Z)=1\quad(|Z|\leq r_*),\qquad
 \rho(Z)=0\quad(|Z|\geq2r_*).
\]
Put
\begin{equation}\label{eq:normalizedgaussian}
 \varphi_T(Z)=\alpha_Te^{-T|Z|^2/2}\rho(Z),\qquad
 \alpha_T^{-2}=\int_{\R^d}e^{-T|Z|^2}\rho(Z)^2\dd Z.
\end{equation}
The normalization is independent of $y$ in orthonormal fiber coordinates, and
\[
 \alpha_T=(T/\pi)^{d/4}(1+O(e^{-cT}))
\]
for some $c>0$. Define the isometric embedding
\begin{equation}\label{eq:gaussianembedding}
 I_T:L^2(M,S_M)\longrightarrow\Hcal,
 \qquad I_Tu(y,Z)=\varphi_T(Z)\pi^*u(y)\otimes1.
\end{equation}
Let $q_0$ be projection onto vertical exterior degree zero. The orthogonal range projection $p_T=I_TI_T^*$ is
\begin{equation}\label{eq:gaussianprojection}
 p_Ts(y,Z)=\varphi_T(Z)\int_{N_y}\varphi_T(Z')q_0s(y,Z')\dd Z'.
\end{equation}
It commutes with every base multiplication operator. The normalization in \eqref{eq:normalizedgaussian} and the product measure \eqref{eq:referenceHilbert} make $I_T$ an isometry. The next lemma estimates $p_Tb$ and $bI_T$ for coefficients bounded by $C|Z|^\gamma$. It also shows that a section orthogonal to the cutoff Gaussian has an exponentially small projection onto the uncut Gaussian.

\begin{lemma}\label{lem:gaussianestimates}
Suppose a multiplication operator $b(y,Z)$ on the model bundle satisfies $|b(y,Z)|\leq C|Z|^\gamma$ on the support of $\rho$, for $\gamma\geq0$. Then $p_Tb$ and $bI_T$, initially defined on $C_c^\infty(N,\mathcal S_0)$ and $C^\infty(M,S_M)$, respectively, extend to bounded operators satisfying
\[
 \norm{p_Tb}\leq C' T^{-\gamma/2},\qquad
 \norm{bI_T}\leq C'T^{-\gamma/2},\qquad T\geq1.
\]
Furthermore, if $p_T^{\mathrm{osc}}$ denotes projection onto the full uncut Gaussian $g_T$ and $p_Tu=0$, then
\begin{equation}\label{eq:gaussiantail}
 \norm{p_T^{\mathrm{osc}}u}\leq Ce^{-cT}\norm u.
\end{equation}
\end{lemma}
\begin{proof}
Both estimates follow from concentration of the Gaussian on the fiber scale $T^{-1/2}$. For the bounds on $p_Tb$ and $bI_T$, Cauchy--Schwarz in \eqref{eq:gaussianprojection} reduces the claim to
\[
 \alpha_T^2\int|Z|^{2\gamma}e^{-T|Z|^2}\rho^2\dd Z\leq C'T^{-\gamma}.
\]
This follows by the change of variables $W=\sqrt T Z$ and $\alpha_T^2=O(T^{d/2})$. The same estimate proves the bound on $bI_T$.

If $p_Tu=0$, then for almost every $y$,
\[
 \int_{N_y}g_T(Z)q_0u(y,Z)\dd Z
 =\int_{N_y}(1-\rho(Z))g_T(Z)q_0u(y,Z)\dd Z.
\]
The $L^2$ norm of $(1-\rho)g_T$ is at most $Ce^{-cT}$ because its support has $|Z|\geq r_*$. A second application of fiberwise Cauchy--Schwarz gives \eqref{eq:gaussiantail}.
\end{proof}
All Gaussian estimates hold for spinor-valued sections by summing over the horizontal factor. The constants are uniform over the compact base.

\subsection{Block decomposition of the deformed operator}
We now decompose $\Acal_T$ along the Gaussian subspace and its orthogonal complement. The diagonal compression will recover $\Dbase$, and the estimates on the other blocks will justify separating these two summands at the level of $K$-homology. Set $p_T^\perp=1-p_T$ and $H'_T=\operatorname{im}p_T^\perp$. On the smooth core, define
\begin{equation}\label{eq:fourblocks}
 \begin{aligned}
 A_{T,1}&=p_T\Acal_Tp_T,& A_{T,2}&=p_T\Acal_Tp_T^\perp,\\
 A_{T,3}&=p_T^\perp\Acal_Tp_T,& A_{T,4}&=p_T^\perp\Acal_Tp_T^\perp.
 \end{aligned}
\end{equation}
Thus the full auxiliary operator has the matrix
\[
 \Acal_T=\begin{pmatrix}A_{T,1}&A_{T,2}\\A_{T,3}&A_{T,4}\end{pmatrix}.
\]
The projection $p_T$ preserves compactly supported smooth sections, so all four expressions are defined on the smooth core. We determine their closed domains after proving that the off-diagonal blocks are bounded.

\subsection{The low-energy block \texorpdfstring{$A_{T,1}$}{A(T,1)}}
\begin{proposition}\label{prop:exactlowenergy}
On smooth sections of $S_M$, one has
\begin{equation}\label{eq:lowenergyidentity}
 I_T^*\Acal_TI_T=\Dbase,
\end{equation}
where $\Dbase$ is the genuine Levi--Civita spin Dirac operator of $g_0^M$. Moreover,
\begin{equation}\label{eq:Jformula}
 \Acal_TI_T=I_T\Dbase+J_T,
 \qquad
 J_Tu=\alpha_Te^{-T|Z|^2/2}c_N(\dd_Z\rho)\pi^*u.
\end{equation}
The map $J_T:L^2(M,S_M)\to H'_T$ is bounded and satisfies
\begin{equation}\label{eq:Jsmall}
 \norm{J_T}\leq Ce^{-cT},\qquad T\geq1.
\end{equation}
\end{proposition}
\begin{proof}
On the image of $I_T$, the horizontal part acts on the base spinor and the vertical part differentiates only the cutoff. Indeed, $I_Tu$ is supported inside the exact-model region, and parallel transport for $\nabla^N$ preserves $\varphi_T$ and the vertical degree-zero vector $1$. Thus
\[
 D^HI_Tu=I_T\Dbase u.
\]
The induced derivative on $u$ is the base Levi--Civita spin connection used in \eqref{eq:modeloperators}. Since the uncut Gaussian lies in the kernel of $Q_T$, we also have
\[
 Q_TI_Tu=\alpha_Te^{-T|Z|^2/2}c_N(\dd_Z\rho)\pi^*u.
\]
The right-hand side has vertical degree one, so its projection by $p_T$ is zero. This proves \eqref{eq:lowenergyidentity} and \eqref{eq:Jformula}.

Finally, $\dd\rho$ is supported where $|Z|\geq r_*$. Therefore
\[
 \alpha_T^2\int|\dd\rho|^2e^{-T|Z|^2}\dd Z\leq Ce^{-cT}.
\]
Taking square roots and renaming $\sqrt C$ and $c/2$ as $C$ and $c$, respectively, gives \eqref{eq:Jsmall}. The estimate is uniform over the base because the fiber integral has the same Euclidean form in every orthonormal trivialization.
\end{proof}
The exact compression \eqref{eq:lowenergyidentity} identifies the closure of the low-energy block with $\Dbase$ under $I_T$. Its domain is $I_TH^1(M,S_M)$, and it represents the classical Dirac $K$-homology class. The remaining term $J_T$ records the cutoff error and determines both off-diagonal blocks.

\subsection{The off-diagonal blocks \texorpdfstring{$A_{T,2}$ and $A_{T,3}$}{A(T,2) and A(T,3)}}
The horizontal model commutes with the Gaussian projection:
\begin{equation}\label{eq:DHcommutes}
 [D^H,p_T]=0.
\end{equation}
To see this, differentiate the fiber integral in \eqref{eq:gaussianprojection} using parallel transport for $\nabla^N$. This transport preserves $\dd Z$, $\varphi_T$, and $q_0$, so the remaining derivative is the connection derivative of the coefficient section.
\begin{proposition}\label{prop:offdiagonal}
The two off-diagonal blocks extend to bounded adjoint operators and satisfy
\begin{equation}\label{eq:offdiagonalexact}
 A_{T,3}=J_TI_T^*,\qquad A_{T,2}=I_TJ_T^*.
\end{equation}
In particular,
\begin{equation}\label{eq:offdiagonalsmall}
 \norm{A_{T,2}}+\norm{A_{T,3}}\leq Ce^{-cT},\qquad T\geq1.
\end{equation}
The commutator $[\Acal_T,p_T]$ is bounded, and $p_T$ preserves $\Dom\Acal_T$.
\end{proposition}
\begin{proof}
Using $p_T=I_TI_T^*$, we can apply \eqref{eq:Jformula} to the definition of $A_{T,3}$ in \eqref{eq:fourblocks}. For a smooth compactly supported section $s$, this gives
\[
 \begin{aligned}
  A_{T,3}s
  &=p_T^\perp\Acal_Tp_Ts
   =p_T^\perp\Acal_TI_T(I_T^*s)\\
  &=p_T^\perp\bigl(I_T\Dbase(I_T^*s)+J_T(I_T^*s)\bigr)
   =J_TI_T^*s.
 \end{aligned}
\]
The last equality uses $p_T^\perp I_T=0$ and $\operatorname{im}J_T\subset H'_T=\operatorname{im}p_T^\perp$.

For $A_{T,2}$, let $s,t$ be smooth compactly supported sections. Since the orthogonal projections $p_T$ and $p_T^\perp$ preserve this core, integration by parts for the formally symmetric operator $\Acal_T$ gives
\[
 \begin{aligned}
  \ip{A_{T,2}s}{t}
  &=\ip{\Acal_Tp_T^\perp s}{p_Tt}\\
  &=\ip{p_T^\perp s}{\Acal_Tp_Tt}\\
  &=\ip{s}{p_T^\perp\Acal_Tp_Tt}
   =\ip{s}{A_{T,3}t}\\
  &=\ip{s}{J_TI_T^*t}
   =\ip{I_TJ_T^*s}{t}.
 \end{aligned}
\]
Since the identity holds for every such $t$, density gives $A_{T,2}s=I_TJ_T^*s$.

By \eqref{eq:Jsmall}, $J_T$ is bounded. Since $I_T$ is an isometry, $\norm{I_T}=\norm{I_T^*}=1$, so both compositions in \eqref{eq:offdiagonalexact} are bounded on $\Hcal$. They therefore give the bounded extensions of the two blocks, and $(J_TI_T^*)^*=I_TJ_T^*$ shows that these extensions are adjoints of one another. Moreover,
\[
 \begin{aligned}
  \norm{A_{T,3}}&\leq\norm{J_T}\norm{I_T^*}
   =\norm{J_T}\leq Ce^{-cT},\\
  \norm{A_{T,2}}&\leq\norm{I_T}\norm{J_T^*}
   =\norm{J_T}\leq Ce^{-cT}.
 \end{aligned}
\]
Adding these bounds and absorbing the factor $2$ into $C$ gives \eqref{eq:offdiagonalsmall}.

To show that $p_T$ preserves $\Dom\Acal_T$, we first express the commutator in terms of the bounded off-diagonal blocks. Using $1=p_T+p_T^\perp$, we obtain on the smooth core
\[
 \begin{aligned}
  [\Acal_T,p_T]
  &=\Acal_Tp_T-p_T\Acal_T\\
  &=p_T^\perp\Acal_Tp_T-p_T\Acal_Tp_T^\perp\\
  &=A_{T,3}-A_{T,2}
   =J_TI_T^*-I_TJ_T^*.
 \end{aligned}
\]
Thus the commutator on the core extends to a bounded operator on $\Hcal$.

Now let $u\in\Dom\Acal_T$. Since the smooth compactly supported sections form a core, we can choose such sections $u_j$ with
\[
 u_j\longrightarrow u,\qquad
 \Acal_Tu_j\longrightarrow\Acal_Tu
 \quad\text{in }\Hcal.
\]
The projection $p_T$ preserves the smooth core, so $p_Tu_j\in\Dom\Acal_T$. Boundedness of $p_T$ gives $p_Tu_j\to p_Tu$, while the commutator identity gives
\[
 \begin{aligned}
  \Acal_Tp_Tu_j
  &=p_T\Acal_Tu_j+(J_TI_T^*-I_TJ_T^*)u_j\\
  &\longrightarrow p_T\Acal_Tu+(J_TI_T^*-I_TJ_T^*)u.
 \end{aligned}
\]
Here both terms converge because $p_T$ and $J_TI_T^*-I_TJ_T^*$ are bounded. Since $\Acal_T$ is closed, convergence of $p_Tu_j$ and $\Acal_Tp_Tu_j$ implies that $p_Tu\in\Dom\Acal_T$ and
\[
 \Acal_Tp_Tu=p_T\Acal_Tu+(J_TI_T^*-I_TJ_T^*)u.
\]
This proves domain preservation and extends the commutator identity from the core to $\Dom\Acal_T$.
\end{proof}
The exponential bound comes from the support of $\dd\rho$, which stays a fixed distance from the zero section. The exact tubular model ensures that this cutoff derivative is the entire off-diagonal error.

\subsection{The complementary block \texorpdfstring{$A_{T,4}$}{A(T,4)}}
We prove that $\norm{A_{T,4}u}^2\geq cT\norm u^2$ for $u\in\Dom A_{T,4}\subset H'_T=(\operatorname{im}I_T)^\perp$ and all sufficiently large $T$, with $c>0$ independent of $T$. We first estimate $\norm{\Acal_Tu}^2$ for smooth compactly supported sections in $H'_T$. Near the zero section, the oscillator estimate supplies the lower bound. Away from it, the positive term $T^2|Z|^2$ dominates the error $TC_*(1+|Z|)$ for large $T$. We combine these two estimates using cutoffs that keep both pieces in $H'_T$, and then use the exponentially small bound on $A_{T,2}$ to obtain the lower bound for $A_{T,4}$. Next, we remove the bounded off-diagonal operator from $\Acal_T$ to identify the self-adjoint domains of the diagonal blocks and prove that they have compact resolvent; this also lets us extend the lower bound from smooth sections to $\Dom A_{T,4}$. Finally, we bound the first derivatives of sections in this domain on each fixed compact set in terms of $\norm{A_{T,4}u}$ and $\sqrt T\norm u$.

\subsubsection{Anticommutation relations}
\begin{lemma}\label{lem:modelanticommutation}
Recall that the model bundle is $\mathcal S_0=\pi^*S_M\hct\Lambda^*\pi^*N\to N$ from \eqref{eq:referenceHilbert}, where $\pi:N\to M$ is the bundle projection and $S_M=S(TM,g_0^M)$ is the base spinor bundle. Its fiber at $(y,Z)\in N$ is $S_{M,y}\hct\Lambda^*N_y$. On $C_c^\infty(N,\mathcal S_0)$, one has
\[
 [D^H,D^N]=0,\qquad [D^H,\hc(Z)]=0.
\]
Consequently, if $u\in C_c^\infty(N,\mathcal S_0)$ is supported in $\{|Z|<4r_*\}\subset N$, the exact-model region where $\Acal_T=D^H+D^N+T\hc(Z)=D^H+Q_T$ by \eqref{eq:exactnormalform}, then
\begin{equation}\label{eq:modelsquareform}
 \norm{\Acal_Tu}^2=\norm{D^Hu}^2+\norm{Q_Tu}^2.
\end{equation}
\end{lemma}
\begin{proof}
We first check that covariant differentiation along a horizontal lift commutes with the vertical operators. Let $X$ be a local vector field on $M$, and choose a local orthonormal frame $u_a$ of $N$, with $\nabla^N_Xu_a=\sum_b\omega_{ba}(X)u_b$. By \eqref{eq:modelhorizontal}, the horizontal lift is $X^H_0=X-\sum_{a,b}\omega_{ab}(X)Z_b\partial_{Z_a}$. Compatibility of $\nabla^0$ with Clifford multiplication gives
\[
 [\nabla^0_{X^H_0},D^N]
 =\sum_{a,b}\bigl(\omega_{ab}(X)+\omega_{ba}(X)\bigr)
 c_N(u_a)\partial_{Z_b}=0.
\]
The two terms come from differentiating the Clifford coefficients in $D^N$ and commuting $X^H_0$ with the vertical derivatives, respectively. They cancel because $\omega(X)$ is skew-symmetric.

Similarly, the tautological vector $Z=\sum_a Z_au_a$ is parallel along $X^H_0$: the coefficient derivatives $X^H_0Z_a=-\sum_b\omega_{ab}(X)Z_b$ cancel the frame derivatives $\nabla^N_Xu_a=\sum_b\omega_{ba}(X)u_b$. Compatibility with the exterior-algebra action therefore gives
\[
 [\nabla^0_{X^H_0},\hc(Z)]=0.
\]
These brackets are ordinary commutators, since the covariant derivative is even.

We now insert the horizontal Clifford factors in $D^H=\sum_i c_M(e_i)\nabla^0_{e_{i,0}^H}$. The graded tensor-product convention makes $c_M(e_i)$ anticommute with $c_N(u_a)$ and $\hc(Z)$, and $c_M(e_i)$ is independent of the fiber coordinates. Hence
\[
 D^HD^N+D^ND^H
 =\sum_i c_M(e_i)[\nabla^0_{e_{i,0}^H},D^N]=0.
\]
The same calculation with $\hc(Z)$ gives $D^H\hc(Z)+\hc(Z)D^H=0$. Since these operators are odd, these are precisely the two graded commutator identities in the statement. In particular, $D^HQ_T+Q_TD^H=0$.

Finally, let $u$ be supported in $\{|Z|<4r_*\}$, where $\Acal_Tu=(D^H+Q_T)u$. Lemma~\ref{lem:modelsymmetric} and pointwise self-adjointness of $\hc(Z)$ imply that $D^H$ and $Q_T$ are formally symmetric. Integration by parts on the compactly supported section $u$ therefore gives
\[
 2\operatorname{Re}\ip{D^Hu}{Q_Tu}
 =\ip{u}{(D^HQ_T+Q_TD^H)u}=0.
\]
Expanding $\norm{(D^H+Q_T)u}^2$ now gives \eqref{eq:modelsquareform}, since its mixed term vanishes.
\end{proof}

\subsubsection{A global coercive estimate on the complementary space}
\begin{proposition}\label{prop:complementgapcore}
There are constants $c_0>0$ and $T_0\geq1$ such that, for $T\geq T_0$ and $u\in C_c^\infty(N,\mathcal S_0)\cap H'_T$,
\begin{equation}\label{eq:fullongap}
 \norm{\Acal_Tu}^2\geq c_0T\norm u^2,
 \qquad
 \norm{A_{T,4}u}^2\geq\frac{c_0}{2}T\norm u^2.
\end{equation}
\end{proposition}
\begin{proof}
We prove the gap near and away from the zero section, and then combine these estimates using cutoffs that preserve orthogonality to the cutoff Gaussian.

\emph{Near the section.}
If $w\in C_c^\infty(N,\mathcal S_0)\cap H'_T$ is supported in $|Z|<4r_*$, Proposition~\ref{prop:oscillator}, Lemma~\ref{lem:gaussianestimates}, and \eqref{eq:modelsquareform} give
\begin{align}
 \norm{\Acal_Tw}^2
 &\geq\norm{Q_Tw}^2
 \geq c_{\mathrm{osc}}T\bigl(\norm w^2-\norm{p_T^{\mathrm{osc}}w}^2\bigr)\nonumber\\
 &\geq c_{\mathrm{osc}}T(1-Ce^{-cT})\norm w^2.
 \label{eq:neargap}
\end{align}
Here \eqref{eq:gaussiantail} converts orthogonality to the cutoff Gaussian into the exponentially small full-Gaussian component needed for the oscillator gap.

\emph{Away from the section.}
By \eqref{eq:auxsquare}--\eqref{eq:auxanticommutator}, at the endpoint $t=1$,
\[
 \norm{\Acal_Tw}^2
 \geq\int_N\bigl(T^2|Z|^2-TC_*(1+|Z|)\bigr)|w|^2\dd\mu_0.
\]
For $|Z|\geq2r_*$, Young's inequality yields
\[
 T^2|Z|^2-TC_*(1+|Z|)
 \geq\frac{T^2}{2}|Z|^2-TC_*-\frac{C_*^2}{2}
 \geq2r_*^2T^2-TC_*-\frac{C_*^2}{2}.
\]
For sufficiently large $T$ this is at least $c_{\mathrm{osc}}T$. Hence
\begin{equation}\label{eq:fargap}
 \norm{\Acal_Tw}^2\geq c_{\mathrm{osc}}T\norm w^2
 \quad\text{if }\supp w\subset\{|Z|\geq2r_*\}.
\end{equation}
This exterior bound follows entirely from the radial potential and \eqref{eq:auxanticommutator}.

\emph{Globalization.}
Choose real radial $\tau_0,\tau_\infty$ with $\tau_0^2+\tau_\infty^2=1$, with $\tau_0=1$ on $|Z|\leq2r_*$ and $\tau_0=0$ on $|Z|\geq3r_*$. Since $\rho$ is supported in $|Z|\leq2r_*$,
\[
 p_T\tau_0=\tau_0p_T=p_T,
 \qquad p_T\tau_\infty=\tau_\infty p_T=0.
\]
Thus both $\tau_0u$ and $\tau_\infty u$ lie in $H'_T$. The first is in the exact-model region, and the second is in the region of \eqref{eq:fargap}.

The principal symbol is Clifford multiplication for the endpoint metric. The zero-order correction and potential commute with scalar functions. Expanding squares and using $\sum_i\tau_i[\Acal_T,\tau_i]=0$ gives the identity
\begin{equation}\label{eq:localizationIMS}
 \norm{\Acal_Tu}^2
 =\sum_{i=0,\infty}\norm{\Acal_T\tau_iu}^2
  -\sum_{i=0,\infty}\norm{[\Acal_T,\tau_i]u}^2.
\end{equation}
The last sum is at most $C_{\mathrm{loc}}\norm u^2$, independently of $T$, because the derivatives of $\tau_i$ have fixed compact support. Inserting \eqref{eq:neargap} and \eqref{eq:fargap}, and increasing $T_0$, proves the first inequality of \eqref{eq:fullongap}.

Finally, for $u\in H'_T$, the components of $\Acal_Tu$ in $p_T\Hcal$ and $H'_T$ are orthogonal. Proposition~\ref{prop:offdiagonal} therefore gives
\[
 \norm{A_{T,4}u}^2
 =\norm{\Acal_Tu}^2-\norm{A_{T,2}u}^2
 \geq(c_0T-Ce^{-cT})\norm u^2.
\]
A final increase of $T_0$ proves the second inequality.
\end{proof}

\subsubsection{Self-adjointness, compact resolvent, and graph domains}
\begin{proposition}\label{prop:blockdomains}
For every fixed $T>0$, the block-diagonal operator
\begin{equation}\label{eq:blockdiagonal}
 A'_T=\begin{pmatrix}A_{T,1}&0\\0&A_{T,4}\end{pmatrix}
\end{equation}
is self-adjoint on $\Dom\Acal_T$ and has compact resolvent. Its summands have domains
\[
 \Dom A_{T,1}=I_TH^1(M,S_M),\qquad
 \Dom A_{T,4}=p_T^\perp\Dom\Acal_T.
\]
For all sufficiently large $T$, there is a constant $c>0$ such that
\begin{equation}\label{eq:complementgap}
 A_{T,4}^2\geq cT\,1_{H'_T}.
\end{equation}
\end{proposition}
\begin{proof}
We obtain the block-diagonal operator by subtracting a bounded self-adjoint perturbation from $\Acal_T$. Namely, the off-diagonal operator
\[
 E_T=A_{T,2}+A_{T,3}=I_TJ_T^*+J_TI_T^*
\]
is bounded and self-adjoint, so $A'_T=\Acal_T-E_T$ is self-adjoint on $\Dom\Acal_T$. The projection $p_T$ preserves this domain and commutes with $A'_T$, so it reduces the operator. The restrictions are self-adjoint on the displayed projected domains. Equation~\eqref{eq:lowenergyidentity} identifies the first domain with that of the base Dirac operator.

Bounded perturbation gives equivalence of the graph norms of $A'_T$ and $\Acal_T$. Since the latter has compact graph-domain inclusion, so does $A'_T$ and each of its restrictions. This proves compact resolvent.

The space $p_T^\perp C_c^\infty(N,\mathcal S_0)$ is a graph core for $A_{T,4}$: apply $p_T^\perp$ to a core approximation and use the bounded commutator from Proposition~\ref{prop:offdiagonal}. It consists of smooth compactly supported sections in $H'_T$. Hence the second inequality in \eqref{eq:fullongap} extends by graph approximation to the whole domain, proving \eqref{eq:complementgap}.
\end{proof}

\subsubsection{Local Sobolev control}
\begin{proposition}\label{prop:localH1}
For every compactly supported smooth cutoff $\chi$ on $N$, there is a constant $C_\chi$, independent of $T\geq T_0$, such that
\begin{equation}\label{eq:localH1}
 \norm{\chi u}_{H^1}
 \leq C_\chi\bigl(\norm{A_{T,4}u}+\sqrt T\norm u\bigr),
 \qquad u\in\Dom A_{T,4}.
\end{equation}
Here the $H^1$ norm measures $\chi u$ and its first covariant derivatives using the fixed model metric $g_{\mathrm{mod}}$ and product connection $\nabla^0$ on $\mathcal S_0$.
\end{proposition}
\begin{proof}
We first prove the estimate for $u\in C_c^\infty(N,\mathcal S_0)\cap H'_T$. The square identity controls $\Dcal_1(\chi u)$ by $\Acal_T(\chi u)$, and the local G\aa rding inequality then controls the first derivatives of $\chi u$.

At $t=1$, \eqref{eq:auxsquare} gives $\Acal_T^2=\Dcal_1^2+T^2|Z|^2+TC_1$. The multiplication operator $C_1=[\Dcal_1,\hc(Z)]$ is bounded on $\supp\chi$ by \eqref{eq:auxanticommutator}, with a bound independent of $T$. Pairing this identity with the compactly supported section $\chi u$ and using formal symmetry gives
\[
 \begin{aligned}
  \norm{\Dcal_1(\chi u)}^2
  &=\norm{\Acal_T(\chi u)}^2-T^2\norm{|Z|\chi u}^2
    -T\ip{C_1\chi u}{\chi u}\\
  &\leq\norm{\Acal_T(\chi u)}^2+C_\chi T\norm{\chi u}^2.
 \end{aligned}
\]
Here we discard the nonpositive second term and bound the last term in absolute value.

The operator $\Dcal_1$ is elliptic and independent of $T$. Proposition~\ref{prop:garding}, applied to $\chi u$, therefore yields
\[
 \begin{aligned}
  \norm{\chi u}_{H^1}
  &\leq C_\chi\bigl(\norm{\Dcal_1(\chi u)}+\norm{\chi u}\bigr)\\
  &\leq C_\chi\bigl(\norm{\Acal_T(\chi u)}+\sqrt T\norm u\bigr).
 \end{aligned}
\]
The second inequality follows by taking square roots in the preceding estimate and using $T\geq T_0\geq1$ to absorb the term $\norm{\chi u}$. The constant remains independent of $T$ because the operator, cutoff, and structures defining the $H^1$ norm are fixed.

It remains to replace $\Acal_T(\chi u)$ by the expression in \eqref{eq:localH1}. Since $u\in H'_T$, we have $\Acal_Tu=A_{T,4}u+A_{T,2}u$, and hence
\[
 \Acal_T(\chi u)
 =\chi A_{T,4}u+\chi A_{T,2}u+[\Dcal_1,\chi]u.
\]
The potential $T\hc(Z)$ commutes with scalar multiplication by $\chi$, so the cutoff commutator is $[\Dcal_1,\chi]$, a bounded multiplication operator independent of $T$. Together with $\norm{A_{T,2}}\leq Ce^{-cT}$, this gives
\[
 \norm{\Acal_T(\chi u)}
 \leq C_\chi\bigl(\norm{A_{T,4}u}+\norm u\bigr).
\]
Substituting into the preceding $H^1$ bound and using $\norm u\leq\sqrt T\norm u$ proves \eqref{eq:localH1} on the smooth core.

For a general $u\in\Dom A_{T,4}$, Proposition~\ref{prop:blockdomains} provides smooth compactly supported sections $u_j\in H'_T$ with $u_j\to u$ and $A_{T,4}u_j\to A_{T,4}u$ in $\Hcal$. For fixed $T$, applying the estimate to $u_j-u_k$ shows that $\chi u_j$ is Cauchy in $H^1$. Its $H^1$ limit is $\chi u$, since it already converges to $\chi u$ in $L^2$. Passing to the limit proves the stated estimate on the full domain.
\end{proof}

The next section uses the block decomposition to identify the deformed class: the Gaussian block is the base Dirac cycle, the bounded off-diagonal terms can be removed by homotopy, and the complementary gap controls the product with the universal flat bundle.

\section{Identification of the deformed class}\label{sec:identification}
We now use the block decomposition of Section~\ref{sec:localization} to identify the deformed class after taking its Kasparov product with the universal flat bundle. Removing the bounded off-diagonal terms leaves the base Dirac cycle and a complementary cycle. The growing spectral gap of the latter makes its coefficient product vanish for large $T$.

Throughout this section, $R_0,\beta,\varepsilon$ are fixed. The blocks belong to the auxiliary operator $\Acal_T$, whose class agrees with that of the deformed sub-Dirac operator $A_T$ by Proposition~\ref{prop:normalization} for every fixed $T>0$.

\subsection{Removing the off-diagonal terms}
\begin{proposition}\label{prop:diagonalcycle}
For every fixed $T>0$, the pair $[\Hcal,\Fc(A'_T)]$ is a Kasparov $(C(M),\C)$-module, where $A'_T$ is given by \eqref{eq:blockdiagonal}.
\end{proposition}
\begin{proof}
By Proposition~\ref{prop:blockdomains}, $A'_T$ is self-adjoint with compact resolvent. It remains to check compactness of the commutators of its bounded transform with base functions. For $f\in C^\infty(M)$, the Gaussian projection $p_T$ commutes with multiplication by $p^*f$, so
\begin{equation}\label{eq:blockcommutator}
 [A_{T,4},f]=p_T^\perp[\Dcal_1,f]p_T^\perp,
 \qquad [A_{T,1},f]=p_T[\Dcal_1,f]p_T.
\end{equation}
The multiplication operator $[\Dcal_1,f]$ is bounded by Lemma~\ref{lem:basecommutators} outside the normalization region and by smoothness on its compact closure. The commutator integral \eqref{eq:boundedcommutator} therefore proves compactness of $[\Fc(A'_T),f]$. Uniform approximation extends this conclusion to $f\in C(M)$. Compactness of $1-\Fc(A'_T)^2$ follows from the compact resolvent.
\end{proof}
\begin{proposition}\label{prop:blockhomotopy}
The diagonal cycle and the cycle defined by $A_T$ have the same class:
\begin{equation}\label{eq:originalsplitting}
 [H_{\beta,\varepsilon},\Fc(A_T)]
 =[L^2(M,S_M),\Fc(\Dbase)]\oplus[H'_T,\Fc(A_{T,4})].
\end{equation}
\end{proposition}
\begin{proof}
We scale the off-diagonal operator $E_T=A_{T,2}+A_{T,3}$ to zero. For $0\leq\eta\leq1$, put
\[
 A_T^\eta=A'_T+\eta E_T
 =\begin{pmatrix}A_{T,1}&\eta A_{T,2}\\
                   \eta A_{T,3}&A_{T,4}\end{pmatrix}.
\]
Since $E_T$ is bounded and self-adjoint, all $A_T^\eta$ are self-adjoint on $\Dom\Acal_T$ and have compact resolvent. Lemma~\ref{lem:boundedperturbation} shows that $\Fc(A_T^\eta)$ depends continuously on $\eta$ in norm and differs from $\Fc(A'_T)$ by a compact operator. The bounded transforms thus form an operator homotopy from the diagonal cycle to the cycle of $\Acal_T$.

The Gaussian embedding $I_T$ identifies the low-energy summand with the base Dirac cycle by \eqref{eq:lowenergyidentity}. Combining this identification with the compact-normalization homotopy \eqref{eq:normalizationclass} gives \eqref{eq:originalsplitting} for the deformed sub-Dirac operator $A_T$.
\end{proof}

\subsection{The complementary class}
To show that the complementary cycle contributes zero to the Rosenberg index, we apply Lemma~\ref{lem:projectionvanish} to its product with the universal flat bundle. This requires two estimates for the bounded transform: a small defect from being an involution and a small commutator with the bundle projection. The spectral gap controls both, as the next lemma shows.
\begin{lemma}\label{lem:gapcommutator}
Let $Q$ be a regular self-adjoint operator on a Hilbert space or Hilbert $C^*$-module, and suppose $Q^2\geq\delta^2$ with $\delta>0$. If an even bounded operator $P$ preserves $\Dom Q$ and $[Q,P]$ is bounded, then
\begin{equation}\label{eq:quantcommutator}
 \norm{[\Fc(Q),P]}\leq
 \frac{2\norm{[Q,P]}}{\sqrt{1+\delta^2}},
 \qquad
 \norm{1-\Fc(Q)^2}\leq\frac1{1+\delta^2}.
\end{equation}
\end{lemma}
\begin{proof}
The gap improves the resolvent bounds in the commutator integral. Set $R_\lambda=(1+\lambda^2+Q^2)^{-1}$ and $a^2=1+\lambda^2$. Rewriting \eqref{eq:boundedcommutator} for $Q$ and $P$, with the commutator in the order $[\Fc(Q),P]$, gives
\[
 [\Fc(Q),P]=\frac2\pi\int_0^\infty
 \bigl(a^2R_\lambda[Q,P]R_\lambda
       -QR_\lambda[Q,P]QR_\lambda\bigr)\,\dd\lambda.
\]
Functional calculus gives
\[
 \norm{R_\lambda}\leq(a^2+\delta^2)^{-1},
 \qquad \norm{QR_\lambda}\leq(a^2+\delta^2)^{-1/2}.
\]
For each $\lambda$, the triangle inequality bounds the operator norm of the expression inside this integral by $\norm{[Q,P]}(a^2\norm{R_\lambda}^2+\norm{QR_\lambda}^2)$, which is at most
\[
 \norm{[Q,P]}\left(\frac{a^2}{(a^2+\delta^2)^2}
                    +\frac1{a^2+\delta^2}\right)
 \leq\frac{2\norm{[Q,P]}}{a^2+\delta^2}.
\]
Integration yields the first inequality in \eqref{eq:quantcommutator}. The second follows directly from $1-\Fc(Q)^2=(1+Q^2)^{-1}$.
\end{proof}

To apply this estimate to the universal flat bundle $E$, choose a smooth projection
\[
 q_E\in M_N(C^\infty(M,C^*\pi_1(M)))
\]
representing $E$. Such a projection is obtained by embedding $E$ isometrically into a finite free module bundle, using finitely many unitary trivializations and a smooth partition of unity. Its pullback $P=p^*q_E$ acts on the coefficient-amplified complementary module $(H'_T\otimes C^*\pi_1(M))^N$. On this module, let
\[
 Q_T=(A_{T,4}\otimes1)^{(N)}.
\]
\begin{lemma}\label{lem:complementprojection}
The operator $Q_T$ is regular self-adjoint with compact resolvent. The projection $P$ preserves its domain, and there is a constant $C_E$, independent of sufficiently large $T$, such that
\[
 \norm{[Q_T,P]}\leq C_E.
\]
Consequently,
\begin{equation}\label{eq:complementnorms}
 \norm{[\Fc(Q_T),P]}\leq\frac{2C_E}{\sqrt{1+cT}},
 \qquad
 \norm{\Fc(Q_T)^2-1}\leq\frac1{1+cT}.
\end{equation}
\end{lemma}
\begin{proof}
By Proposition~\ref{prop:blockdomains}, $A_{T,4}$ is self-adjoint with compact resolvent. We obtain the corresponding properties of $Q_T$ by tensoring the scalar resolvents with the coefficient identity. More precisely, $Q_T$ is the closure of $(A_{T,4}\otimes1)^{(N)}$ on finite sums of elementary tensors whose scalar factors lie in $\Dom A_{T,4}$. This is a densely defined symmetric operator. Set
\[
 R_\pm=\bigl((A_{T,4}\pm i)^{-1}\otimes1\bigr)^{(N)}.
\]
These operators are bounded and adjointable. On finite sums of elementary tensors, the scalar resolvent identities give $Q_TR_\pm v=v\mp iR_\pm v$. For an arbitrary vector $v$ in the coefficient module, choose such sums $v_k\to v$. Then
\[
 R_\pm v_k\longrightarrow R_\pm v,
 \qquad
 Q_TR_\pm v_k=v_k\mp iR_\pm v_k
       \longrightarrow v\mp iR_\pm v.
\]
Closedness of $Q_T$ implies that $R_\pm v\in\Dom Q_T$ and $(Q_T\pm i)R_\pm v=v$. To prove the reverse identity, let $u\in\Dom Q_T$. By the definition of $Q_T$ as a closure, there are finite sums of elementary tensors $u_k$, with scalar factors in $\Dom A_{T,4}$, such that $u_k\to u$ and $Q_Tu_k\to Q_Tu$. The scalar resolvent identity gives $R_\pm(Q_T\pm i)u_k=u_k$ for every $k$. Since $R_\pm$ is bounded, we can pass to the limit:
\[
 R_\pm(Q_T\pm i)u
 =\lim_{k\to\infty}R_\pm(Q_T\pm i)u_k
 =\lim_{k\to\infty}u_k=u.
\]
In particular, both $Q_T+i$ and $Q_T-i$ are surjective. The resolvent criterion for closed densely defined symmetric Hilbert-module operators now shows that $Q_T$ is self-adjoint and regular.

The operators $R_\pm$ are also compact in the sense of Hilbert modules. Indeed, since $C^*\pi_1(M)$ is unital, a scalar rank-one operator determined by $\xi,\eta\in H'_T$ becomes the Hilbert-module rank-one operator determined by $\xi\otimes1,\eta\otimes1$ after tensoring with the coefficient identity. Approximating the scalar resolvents by finite-rank operators, and taking the finite matrix amplification, proves the assertion. The resolvent identities above also identify the functional calculus of $Q_T$ with the tensor product of the scalar functional calculus. Hence \eqref{eq:complementgap} gives $Q_T^2\geq cT$ for sufficiently large $T$.

We compute the commutator with $P$ on smooth compactly supported coefficient-valued sections in the complementary summand. The projection $P=p^*q_E$ is constant along each fiber and acts only on the coefficient factors. It therefore commutes with the Gaussian projection and preserves this smooth core. The potential $T\hc(Z)$ also commutes with $P$. Applying \eqref{eq:blockcommutator} entrywise gives, on the core,
\[
 [Q_T,P]=
 (p_T^\perp\otimes1)^{(N)}
 [ (\Dcal_1\otimes1)^{(N)},p^*q_E]
 (p_T^\perp\otimes1)^{(N)}.
\]
Choose a local $g_1$-orthonormal frame $e_j$ of $TN$, and let $c(e_j)$ denote Clifford multiplication transported to $\mathcal S_0$ by $\mathcal J_1$. For a smooth coefficient-valued section $s$, the Leibniz rule gives
\[
 [ (\Dcal_1\otimes1)^{(N)},p^*q_E]s
 =\sum_j(c(e_j)\otimes1)^{(N)}
       \bigl(e_j(p^*q_E)s\bigr).
\]
Here $e_j(p^*q_E)$ is obtained by differentiating each matrix entry in the direction $e_j$. The terms differentiating $s$ cancel, while the zeroth-order terms of $\Dcal_1$ commute with $p^*q_E$. Thus the middle operator in the compressed formula is multiplication by the displayed matrix of bundle endomorphisms.

For the fixed parameters $R_0,\beta,\varepsilon$, its norm is bounded independently of $T$. Outside the compact normalization region $\{|Z|\leq6r_*\}$, Lemma~\ref{lem:basecommutators} bounds the derivatives $e_j(p^*q_E)$. The same proof applies to coefficient-valued entries, since the derivatives of $q_E$ are bounded in the coefficient norm on the compact base. On $\{|Z|\leq6r_*\}$, smoothness and compactness give the bound. Each Clifford factor has norm one, and there are finitely many terms in the sum. Neither $\Dcal_1$ nor $P$ depends on $T$, and the two orthogonal projections in the compressed formula have norm at most one. Consequently the core commutator extends to a bounded adjointable operator, still denoted by $[Q_T,P]$, satisfying
\[
 \norm{[Q_T,P]}\leq C_E.
\]
The constant $C_E$ may depend on $q_E$, the fixed parameters, and the normalization data, but not on $T$.

To pass from the core to the full domain, let $u\in\Dom Q_T$ and choose core sections $u_k$ such that $u_k\to u$ and $Q_Tu_k\to Q_Tu$. Since $P$ preserves the core,
\[
 Pu_k\longrightarrow Pu,
 \qquad
 Q_TPu_k=PQ_Tu_k+[Q_T,P]u_k
       \longrightarrow PQ_Tu+[Q_T,P]u.
\]
Closedness of $Q_T$ therefore gives $Pu\in\Dom Q_T$ and
\[
 Q_TPu=PQ_Tu+[Q_T,P]u.
\]
This proves domain preservation and identifies the commutator on $\Dom Q_T$ with its bounded extension. We may now apply Lemma~\ref{lem:gapcommutator} with $\delta^2=cT$ and $\norm{[Q_T,P]}\leq C_E$, which gives \eqref{eq:complementnorms}.
\end{proof}
\begin{proposition}\label{prop:complementzero}
For sufficiently large $T$, the Kasparov product of \eqref{eq:bundleclass} with $[H'_T,\Fc(A_{T,4})]$, after exterior tensoring by the coefficient algebra, is zero in $KK(\C,C^*\pi_1(M))$.
\end{proposition}
\begin{proof}
The two bounds in \eqref{eq:complementnorms} make the compressed bounded transform invertible. Indeed, their sum is less than one for all sufficiently large $T$, so Lemma~\ref{lem:projectionvanish} applies to the coefficient-amplified complementary cycle and $P$. It identifies the product with $[P(H'_T\otimes C^*\pi_1(M))^N,P\Fc(Q_T)P]$ and proves that this class vanishes in $KK(\C,C^*\pi_1(M))$.
\end{proof}

\subsection{Identification with the Rosenberg index}
\begin{proposition}\label{prop:rosenbergidentification}
The Kasparov product of \eqref{eq:bundleclass}, \eqref{eq:bottclass}, and \eqref{eq:totalclass}, with the coefficient exterior products understood as in \eqref{eq:triple}, is the Rosenberg index of $M$.
\end{proposition}
\begin{proof}
After taking the product with the universal flat bundle, the splitting \eqref{eq:originalsplitting} leaves only the base Dirac contribution. To see this, fix $R_0,\beta,\varepsilon$ and choose $T$ large enough for Proposition~\ref{prop:complementzero}. Proposition~\ref{prop:product} represents $b\hct d_{\M}$ by the cycle defined by $A_T$, and Proposition~\ref{prop:blockhomotopy} splits its class into the base Dirac and complementary cycles. Their products with the universal flat bundle are respectively $\alpha(M)$, by definition, and zero, by Proposition~\ref{prop:complementzero}.

For each fixed choice of $R_0,\beta,\varepsilon$, Proposition~\ref{prop:product} shows that the class represented by $A_T$ is independent of $T>0$. Hence the identification obtained above for sufficiently large $T$ holds for every $T>0$.
\end{proof}

\section{Vanishing under leafwise positive scalar curvature}\label{sec:vanishing}
Assume throughout this section that the leafwise scalar curvature is positive. We prove that the product of the universal flat bundle with the cycle defined by $A_T$ vanishes, which, together with Proposition~\ref{prop:rosenbergidentification}, gives the even-dimensional case of the main theorem.

The proof combines two sources of positivity: leafwise scalar curvature near the distinguished section and the radial potential away from it. We first choose the parameters so that these estimates give a global spectral gap. We then identify the coefficient product with the operator twisted by the given flat connection, for which the same gap holds because the twisting curvature is zero. All curvature estimates use the rescaled metric \eqref{eq:rescaledmetric}, the full sub-Dirac operator \eqref{eq:fullsubdirac}, and the radial potential $T\hc(V)$.

\subsection{Choice of parameters and the estimate at infinity}
Let
\[
 \kappa=\min_{x\in M}k^F(x)>0.
\]
We choose radial cutoffs to combine the near-section and exterior estimates, keeping track of the resulting derivative loss before fixing the radius. Fix a smooth $\vartheta:[0,\infty)\to[0,\pi/2]$ which is zero on $[0,1]$ and $\pi/2$ on $[2,\infty)$. Put $\tau_0=\cos\vartheta$ and $\tau_\infty=\sin\vartheta$, so their squares sum to one. By the radial derivative estimate in Section~\ref{sec:geometry}, the number
\[
 C_{\mathrm{cut}}=(L_F^2+1)\norm{\vartheta'}_\infty^2
\]
is independent of $R_0,\beta,\varepsilon,T$ and gives, for $\sigma_i=\tau_i(\ell/R_0)$,
\[
 |\dd\sigma_0|^2_{\beta,\varepsilon}
 +|\dd\sigma_\infty|^2_{\beta,\varepsilon}
 \leq\frac{C_{\mathrm{cut}}}{R_0^2\beta^2}.
\]

The constant $\theta$ in \eqref{eq:theta} is independent of the cutoff radius. We can therefore use it to choose first $a$ and then $R_0$ so that
\begin{equation}\label{eq:parameteraradius}
 a=\frac{32\theta}{\kappa},\qquad
 R_0\geq\max\left\{2,\ a\sqrt\kappa,
                      \sqrt{\frac{32C_{\mathrm{cut}}}{\kappa}}\right\}.
\end{equation}
Fix the transverse metric \eqref{eq:capmetric} with this $R_0$. Set $T=1/(a\beta)$. For each fixed $\beta,\varepsilon$, Propositions~\ref{prop:product} and \ref{prop:rosenbergidentification} identify the product of the universal flat bundle with the cycle defined by $A_T$ as $\alpha(M)$ for every $T>0$.

The square identity \eqref{eq:square} for $A_T$ and the anticommutator estimate \eqref{eq:theta} give the zeroth-order lower bound
\begin{align}
 T^2\ell^2+T\sum_i c_{\beta,\varepsilon}(f_i)\hc(\nabla^{\Ftwo}_{f_i}V)
 &\geq\frac1{\beta^2}\left(\frac{\ell^2}{a^2}
                         -\frac{\theta\beta\ell}{a}-\frac\theta a\right)
 \nonumber\\
 &\geq\frac1{\beta^2}\left(\frac{\ell^2}{2a^2}
                         -\frac{\theta^2\beta^2}{2}-\frac\theta a\right).
 \label{eq:faroriginal}
\end{align}
For $\ell\geq R_0$ and $\theta^2\beta^2/2\leq\kappa/32$, the bracket is at least
\[
 \frac\kappa2-\frac\kappa{32}-\frac\kappa{32}
 =\frac{14\kappa}{32}.
\]
For sections supported in $\{\ell\geq R_0\}$, the square $\Dsub^2$ contributes a nonnegative quadratic form. Dropping it gives the exterior lower bound for $A_T^2$ using only the radial potential and its anticommutator.

\subsection{The estimate near the distinguished section}
On $\M_{2R_0}=\{\ell\leq2R_0\}$, the transverse metric is the canonical almost-isometric metric. Here leafwise scalar curvature controls the sub-Dirac square through Zhang's rescaled Lichnerowicz formula \cite[Equations~(2.17) and~(2.28), following~(1.71)]{Zha17}:
\[
 \Dsub^2=-\Delta^{\beta,\varepsilon}
          +\frac{p^*k^F}{4\beta^2}+\mathcal R_{\beta,\varepsilon,R_0},
 \qquad
 \norm{\mathcal R_{\beta,\varepsilon,R_0}}
 \leq C_{R_0}\left(\frac1\beta+\frac{\varepsilon^2}{\beta^2}\right).
\]
Here $-\Delta^{\beta,\varepsilon}$ is the nonnegative connection Laplacian, the remainder is a self-adjoint zeroth-order endomorphism, and $C_{R_0}$ is independent of $0<\beta,\varepsilon\leq1$ once $R_0$ has been fixed.

Combining this formula with \eqref{eq:square} and the anticommutator bound \eqref{eq:theta}, we obtain a lower bound for $A_T^2$ near the section. For sections supported in $\M_{2R_0}$, use $\ell\leq2R_0$, drop $T^2\ell^2\geq0$, and substitute $T=1/(a\beta)$ to get
\begin{equation}\label{eq:nearoriginal}
 A_T^2\geq\frac1{\beta^2}
 \left[\frac\kappa4-\frac\theta a-C_{R_0}\varepsilon^2
       -\beta\left(C_{R_0}+\frac{2\theta R_0}{a}\right)\right].
\end{equation}
The parameters $a$ and $R_0$ are already fixed, with $\theta/a=\kappa/32$. Choose $0<\varepsilon_0\leq1$ and then $0<\beta_0\leq1$ sufficiently small that, for $\varepsilon=\varepsilon_0$ and $0<\beta\leq\beta_0$, the two remaining error terms in the bracket in \eqref{eq:nearoriginal} are each at most $\kappa/32$. This gives $A_T^2\geq5\kappa/(32\beta^2)$ on sections supported in $\M_{2R_0}$. Decreasing $\beta_0$ if necessary also ensures $\theta^2\beta^2/2\leq\kappa/32$, so the exterior bound $A_T^2\geq14\kappa/(32\beta^2)$ from \eqref{eq:faroriginal} holds for the same parameters.

It remains to combine the two supported estimates. The real scalar cutoffs $\sigma_0$ and $\sigma_\infty$ are supported in $\M_{2R_0}$ and $\{\ell\geq R_0\}$, respectively. Expand
\[
 A_T(\sigma_i u)=\sigma_iA_Tu+[A_T,\sigma_i]u.
\]
The cross terms cancel on summing because $\sum_i\sigma_i[A_T,\sigma_i]=0$. Thus
\[
 \norm{A_Tu}^2=\sum_{i=0,\infty}\norm{A_T\sigma_i u}^2
              -\sum_{i=0,\infty}\norm{[A_T,\sigma_i]u}^2.
\]
The derivative contribution is at most $\kappa\norm u^2/(32\beta^2)$ by \eqref{eq:parameteraradius}. Combining the near and far estimates yields
\begin{equation}\label{eq:originalgap}
 \left( D_{S(\Ftwo),\beta,\varepsilon_0}
              +\frac{\hc(V)}{a\beta}\right)^2
 \geq\frac\kappa{8\beta^2}\,I.
\end{equation}

\subsection{Vanishing after twisting by the universal flat bundle}
\begin{proposition}\label{prop:flatvanishing}
Let $H=L^2(\M,S(\LF\oplus\Fone)\hct\Lambda^*\Ftwo)$. For $0<\beta\leq\beta_0$, the Kasparov product of \eqref{eq:bundleclass} with
\begin{equation}\label{eq:positivecycle}
 \left[H,\Fc\left(D_{S(\Ftwo),\beta,\varepsilon_0}
                         +\frac{\hc(V)}{a\beta}\right)\right]
\end{equation}
is zero in $KK(\C,C^*\pi_1(M))$.
\end{proposition}
\begin{proof}
Write $\B=C^*\pi_1(M)$. Choose the smooth projection $q_E\in M_N(C^\infty(M,\B))$ representing $E$ and put $P=p^*q_E$. By Lemma~\ref{lem:projectionvanish}, the product is represented by
\begin{equation}\label{eq:coefficientcompression}
 [PH_0,P\Fc(\mathbb D)P],
 \quad H_0=(H\otimes\B)^N,
 \quad\mathbb D=
 \left(\left(D_{S(\Ftwo),\beta,\varepsilon_0}
                    +\frac{\hc(V)}{a\beta}\right)\otimes1_\B\right)^{(N)}.
\end{equation}
To realize this product by a differential operator, we compare the connection induced by $q_E$ with the given flat connection on $E$.

\emph{The connection difference.}
The Grassmann connection $\nabla^{\mathrm{Gr}}=q_E\dd q_E$ on $E$ differs from the given flat unitary connection by
\[
 \omega=\nabla^E-\nabla^{\mathrm{Gr}}
 \in\Omega^1(M,\End_\B(E)).
\]
On the total space, the tensor-product connections satisfy
\begin{equation}\label{eq:coefficientconnection}
 \nabla^{\Ecal\otimes p^*E}
 -P\bigl(\nabla^{\Ecal}\otimes1+1\otimes\dd\bigr)P
 =1\otimes p^*\omega.
\end{equation}
Here $\Ecal=S(\LF\oplus\Fone)\hct\Lambda^*\Ftwo$, and all Clifford actions and spin connections use the fixed parameters $\beta,\varepsilon_0$. The range of $P$ identifies $PH_0$ with $L^2(\M,\Ecal\otimes p^*E)$.

\emph{Regular self-adjointness and the common domain.}
By the tensor-product resolvent argument in the proof of Lemma~\(\ref{lem:complementprojection}\), the operator \(\mathbb D\) is regular self-adjoint and has compact resolvent. The potential commutes with $P$, so Lemma~\ref{lem:basecommutators} gives a bounded extension of the commutator $[\mathbb D,P]$ from the smooth compactly supported core to $H_0$.

Since $P$ preserves the smooth compactly supported core and $[\mathbb D,P]$ is bounded, the approximation argument in the proof of Lemma~\ref{lem:complementprojection} shows that $P$ preserves $\Dom\mathbb D$ and
\[
 \mathbb D(Pu)=P\mathbb D u+[\mathbb D,P]u,
 \qquad u\in\Dom\mathbb D.
\]

Define, on this domain,
\begin{equation}\label{eq:coefficientdiagonal}
 \mathbb D_0=P\mathbb DP+(1-P)\mathbb D(1-P).
\end{equation}
The difference $\mathbb D-\mathbb D_0=P\mathbb D(1-P)+(1-P)\mathbb DP$ is bounded and self-adjoint. Since bounded self-adjoint perturbations preserve regular self-adjointness and compact resolvent, $\mathbb D_0$ is regular self-adjoint on $\Dom\mathbb D$ and has compact resolvent. Since $P$ preserves $\Dom\mathbb D_0$ and commutes with $\mathbb D_0$ on this domain, the restriction of $\mathbb D_0$ to $PH_0$, given by
\[
 \mathbb D_P=P\mathbb DP,
 \qquad \Dom\mathbb D_P=P\Dom\mathbb D,
\]
is regular self-adjoint with compact resolvent on $PH_0$.

Let the genuine twisted operator on $PH_0$ be
\[
 \mathbb D^E=D^E_{S(\Ftwo),\beta,\varepsilon_0}
                    +\frac{\hc(V)}{a\beta},
 \qquad
 D^E_{S(\Ftwo),\beta,\varepsilon_0}
 =\sum_i c_{\beta,\varepsilon_0}(f_i)
                      \nabla^{\Ecal\otimes p^*E}_{f_i}.
\]
Equation~\eqref{eq:coefficientconnection} gives on the smooth compactly supported core
\begin{equation}\label{eq:coefficientdifference}
 \mathbb D^E-\mathbb D_P
 =\sum_i c_{\beta,\varepsilon_0}(f_i)(p^*\omega)(f_i)=K_\omega.
\end{equation}
The pullback one-form $p^*\omega$ vanishes on vertical vectors. Its horizontal contractions are bounded for the fixed parameters by the metric modification at infinity and compactness of the base. Since both connections are unitary, $K_\omega$ is self-adjoint; it is also odd. Thus $\mathbb D^E$ is regular self-adjoint on $\Dom\mathbb D_P$ and has compact resolvent.

\emph{Identification of the product representative.}
For regular self-adjoint $R,S$ with common domain and bounded difference, Lemma~\ref{lem:boundedperturbation} gives the norm-convergent formula
\begin{equation}\label{eq:coefficientfunctional}
 \begin{split}
 \Fc(R)-\Fc(S)=\frac2\pi\int_0^\infty
 (1+\lambda^2+R^2)^{-1}
 \cdot\bigl((1+\lambda^2)(R-S)+R(S-R)S\bigr)\\
 \cdot(1+\lambda^2+S^2)^{-1}\,\dd\lambda.
 \end{split}
\end{equation}
When both resolvents are compact, the difference is compact. Apply this first to $\mathbb D,\mathbb D_0$ and then to $\mathbb D_P,\mathbb D^E$. Functional calculus respects the decomposition of $\mathbb D_0$ into its restrictions to $PH_0$ and $(1-P)H_0$, so
\[
 P\Fc(\mathbb D_0)P=\Fc(\mathbb D_P).
\]
Therefore
\[
 P\Fc(\mathbb D)P-\Fc(\mathbb D^E)
 =P(\Fc(\mathbb D)-\Fc(\mathbb D_0))P
     +\Fc(\mathbb D_P)-\Fc(\mathbb D^E)
 \in\K_\B(PH_0).
\]
It follows that the product in \eqref{eq:coefficientcompression} is represented by
\begin{equation}\label{eq:truetwistedcycle}
 [L^2(\M,\Ecal\otimes p^*E),\Fc(\mathbb D^E)].
\end{equation}

\emph{Positivity and vanishing.}
Since $E$ is flat, the additional curvature term in the twisted Lichnerowicz formula vanishes:
\[
 \frac12\sum_{i,j}c(f_i)c(f_j)R^{p^*E}(f_i,f_j)=0.
\]
The remaining curvature endomorphisms and the commutator with the radial Clifford potential are the untwisted ones tensored with $1_E$. The connection Laplacian remains nonnegative with respect to the $\B$-valued inner product. Therefore the arguments for \eqref{eq:nearoriginal} and \eqref{eq:faroriginal} apply to $\mathbb D^E$, with the same parameters and with inequalities understood in the positive cone of $\B$. They give lower bounds of $5\kappa/(32\beta^2)$ and $14\kappa/(32\beta^2)$ for smooth compactly supported sections in $\{\ell\leq2R_0\}$ and $\{\ell\geq R_0\}$, respectively.

The scalar cutoffs $\sigma_0,\sigma_\infty$ commute with the coefficient action, and
$[\mathbb D^E,\sigma_i]=c(\dd\sigma_i)\otimes1_E$.
Thus the calculation leading to \eqref{eq:originalgap} applies with $\B$-valued inner products. By \eqref{eq:parameteraradius}, the cutoff error is at most $\kappa\ip uu_\B/(32\beta^2)$. Subtracting this error from the smaller of the two lower bounds gives
\begin{equation}\label{eq:coefficientgap}
 \ip{\mathbb D^Eu}{\mathbb D^Eu}_{\B}
 \geq\frac\kappa{8\beta^2}\ip uu_{\B}
\end{equation}
for every smooth compactly supported section $u$. Approximation in the graph norm of $\mathbb D^E$ extends \eqref{eq:coefficientgap} to every $u\in\Dom\mathbb D^E$.

Put $\delta^2=\kappa/(8\beta^2)$. To see that \eqref{eq:coefficientgap} excludes spectrum in $(-\delta,\delta)$, suppose otherwise and choose $f\in C_c((-\delta,\delta))$ with $f(\mathbb D^E)\ne0$. Choose $v\in PH_0$ such that $u=f(\mathbb D^E)v\ne0$. Since $f$ has compact support, $u\in\Dom\mathbb D^E$, and there is $\eta>0$ such that $x^2-\delta^2\leq-\eta$ on $\supp f$. Functional calculus and \eqref{eq:coefficientgap} then give
\[
 0\leq\ip{\mathbb D^Eu}{\mathbb D^Eu}_\B-\delta^2\ip uu_\B
 \leq-\eta\ip uu_\B,
\]
contradicting $u\ne0$. Hence $\mathbb D^E$ is invertible, with
\[
 \norm{(\mathbb D^E)^{-1}}\leq\frac{2\sqrt2\,\beta}{\sqrt\kappa}.
\]

On the spectrum of $\mathbb D^E$, the function $(\Fc-\sgn)$ is continuous and vanishes at infinity. Since $\mathbb D^E$ has compact resolvent, functional calculus gives
\[
 \Fc(\mathbb D^E)-\sgn(\mathbb D^E)\in\K_\B(PH_0).
\]
Thus replacing $\Fc(\mathbb D^E)$ by $\sgn(\mathbb D^E)$ does not change the class in \eqref{eq:truetwistedcycle}. The operator $\sgn(\mathbb D^E)$ is odd and self-adjoint, its square is the identity, and it commutes with the scalar left action of $\C$. The resulting Kasparov cycle is therefore degenerate, so the class in \eqref{eq:truetwistedcycle}, and hence the required product, is zero.
\end{proof}
\section{Proof of the main theorem}\label{sec:mainproof}
\begin{proof}[Proof of Theorem~\ref{thm:main}]
By Remark~\ref{rem:even} and Theorem~\ref{thm:rosclassic}, it suffices to treat the case where $M$ is even-dimensional and $\rank(TM/F)>0$. Assume that $M$ is connected. Choose the geometric parameters as in Section~\ref{sec:vanishing}. Proposition~\ref{prop:rosenbergidentification} identifies the product of \eqref{eq:bundleclass}, \eqref{eq:bottclass}, and \eqref{eq:totalclass} with $\alpha(M)$. By Proposition~\ref{prop:product}, the same class is the product of \eqref{eq:bundleclass} with the deformed sub-Dirac cycle \eqref{eq:positivecycle}. Proposition~\ref{prop:flatvanishing} proves that this product is zero. Hence $\alpha(M)=0$.
\end{proof}

\section*{Acknowledgments}
The authors thank Professor Weiping Zhang for suggesting this problem and for many enlightening discussions.

G. Su was partially supported by National Natural Science Foundation of China (Grant Nos.\ 12425106 and 12621005), Natural Science Foundation of Tianjin, China (Grant No.\ 25JCZDJC01110), Fundamental Research Funds for the Central Universities, Grant No.\ 050-63261176 and Nankai Zhide Foundation.  Z. Yi was partially supported by the National Key R\&D Program of China (2022YFA100700).


\begin{thebibliography}{99}
\bibitem[Aub98]{Aub98}
Thierry Aubin, \emph{Some nonlinear problems in Riemannian geometry}, Springer Monographs in Mathematics, Springer-Verlag, Berlin, 1998.
\bibitem[BL91]{BL91}
Jean-Michel Bismut and Gilles Lebeau, \emph{Complex immersions and Quillen metrics}, Inst. Hautes \'Etudes Sci. Publ. Math. \textbf{74} (1991), ii+298 pp.
\bibitem[Bla98]{Bla98}
Bruce Blackadar, \emph{$K$-theory for operator algebras}, second edition, Mathematical Sciences Research Institute Publications, vol.~5, Cambridge University Press, Cambridge, 1998.
\bibitem[Con86]{Con86}
Alain Connes, \emph{Cyclic cohomology and the transverse fundamental class of a foliation}, in Geometric methods in operator algebras (Kyoto, 1983), Pitman Research Notes in Mathematics Series, vol.~123, Longman Scientific \& Technical, Harlow, 1986, pp.~52--144.
\bibitem[Gro96]{Gro96}
Mikhael Gromov, \emph{Positive curvature, macroscopic dimension, spectral gaps and higher signatures}, in Functional analysis on the eve of the 21st century, Vol.~II (New Brunswick, NJ, 1993), Progress in Mathematics, vol.~132, Birkh\"auser Boston, Boston, MA, 1996, pp.~1--213.
\bibitem[Hel78]{Hel78}
Sigurdur Helgason, \emph{Differential geometry, Lie groups, and symmetric spaces}, Pure and Applied Mathematics, vol.~80, Academic Press, New York--London, 1978.
\bibitem[HR00]{HR00}
Nigel Higson and John Roe, \emph{Analytic $K$-homology}, Oxford Mathematical Monographs, Oxford University Press, Oxford, 2000.
\bibitem[Les05]{Les05}
Matthias Lesch, \emph{The uniqueness of the spectral flow on spaces of unbounded self-adjoint Fredholm operators}, in Spectral geometry of manifolds with boundary and decomposition of manifolds, Contemporary Mathematics, vol.~366, American Mathematical Society, Providence, RI, 2005, pp.~193--224.
\bibitem[Lic63]{Lic63}
Andr\'e Lichnerowicz, \emph{Spineurs harmoniques}, C. R. Acad. Sci. Paris \textbf{257} (1963), 7--9.
\bibitem[LS19]{LS19}
Yiannis Loizides and Yanli Song, \emph{Quantization of Hamiltonian loop group spaces}, Math. Ann. \textbf{374} (2019), nos.~1--2, 681--722.
\bibitem[LZ01]{LZ01}
Kefeng Liu and Weiping Zhang, \emph{Adiabatic limits and foliations}, in Topology, geometry, and algebra: interactions and new directions (Stanford, CA, 1999), Contemporary Mathematics, vol.~279, American Mathematical Society, Providence, RI, 2001, pp.~195--208.
\bibitem[Ros83]{Ros83}
Jonathan Rosenberg, \emph{$C^*$-algebras, positive scalar curvature, and the Novikov conjecture}, Inst. Hautes \'Etudes Sci. Publ. Math. \textbf{58} (1983), 197--212.
\bibitem[Ros86]{Ros86}
Jonathan Rosenberg, \emph{$C^*$-algebras, positive scalar curvature, and the Novikov conjecture. III}, Topology \textbf{25} (1986), no.~3, 319--336.
\bibitem[SY25]{SY25}
Guangxiang Su and Zelin Yi, \emph{Enlargeable foliations and the monodromy groupoid}, J. Noncommut. Geom. \textbf{19} (2025), no.~4, 1287--1321.
\bibitem[Tes14]{Tes14}
Gerald Teschl, \emph{Mathematical methods in quantum mechanics: With applications to Schr\"odinger operators}, second ed., Graduate Studies in Mathematics, vol.~157, American Mathematical Society, Providence, RI, 2014.
\bibitem[YZ18]{YZ18}
Jianqing Yu and Weiping Zhang, \emph{Positive scalar curvature and the Euler class}, J. Geom. Phys. \textbf{126} (2018), 193--203.
\bibitem[Zha01]{Zha01}
Weiping Zhang, \emph{Lectures on Chern--Weil theory and Witten deformations}, Nankai Tracts in Mathematics, vol.~4, World Scientific, River Edge, NJ, 2001.
\bibitem[Zha17]{Zha17}
Weiping Zhang, \emph{Positive scalar curvature on foliations}, Ann. of Math. (2) \textbf{185} (2017), no.~3, 1035--1068.
\end{thebibliography}
\end{document}